\documentclass[a4paper]{article}

\usepackage{graphicx}
\usepackage{amsmath}
\usepackage{amsfonts}
\usepackage{textcomp}
\usepackage{authblk}
\usepackage{amssymb}
\usepackage{epstopdf}
\usepackage{mathptmx} 
\usepackage[ansinew]{inputenc}
\usepackage[english]{babel}
\usepackage{latexsym}
\usepackage{subfigure} 
\usepackage{multirow}
\usepackage{color}
\usepackage{float}
\usepackage[colorlinks=true,linkcolor=blue,citecolor=blue]{hyperref}
\usepackage{amsfonts,amsmath,amsthm,mathtools,bbm}

\newtheorem{proposition}{Proposition}
\newtheorem{lemma}{Lemma}
\newtheorem{theorem}{Theorem}
\theoremstyle{remark}\newtheorem{remark}{Remark}

 \newcommand{\RR}{\mathbb{R}}

 \newcommand{\F}{\mathcal{F}}
 \newcommand{\B}{\mathcal{B}}
 \newcommand{\C}{\mathcal{C}}
\newcommand{\be}{\begin{equation}}
\newcommand{\ee}{\end{equation}}

\begin{document}

\title{Structure preserving schemes for nonlinear Fokker-Planck equations and applications}
\author{	Lorenzo Pareschi \thanks{Department of Mathematics and Computer Sciences, University of Ferrara, Via N. Machiavelli 35, 44121 Ferrara, Italy.
			\texttt{lorenzo.pareschi@unife.it}} \and
		Mattia Zanella\thanks{Department of Mathematical Sciences ``G. L. Lagrange'', Politecnico di Torino, Corso Duca degli Abruzzi 24, 10129 Turin, Italy.
			\texttt{mattia.zanella@polito.it}}
		}
\date{}

\maketitle

\begin{abstract}
In this paper we focus on the construction of numerical schemes for nonlinear Fokker-Planck equations that preserve the structural properties, like non negativity of the solution, entropy dissipation and large time behavior. The methods here developed are second order accurate, they do not require any restriction on the mesh size and are capable to capture the asymptotic steady states with arbitrary accuracy. These properties are essential for a correct description of the underlying physical problem. Applications of the schemes to several nonlinear Fokker-Planck equations with nonlocal terms describing emerging collective behavior in socio-economic and life sciences are presented. 


\end{abstract}

\section{Introduction}

In this paper we construct and discuss a steady-state preserving method for a wide class of nonlinear Fokker-Planck equations of the form 
\begin{equation}\begin{cases}\label{eq:NAD}
\partial_t f(w,t) = \nabla_w \cdot \Big[ \B[f](w,t)f(w,t) + \nabla_w (D(w)f(w,t)) \Big],\\
f(w,0)=f_0(w),
\end{cases}\end{equation}
where $t\ge 0$, $w\in \Omega\subseteq \RR^d$, $d\ge 1$, $f(w,t)\ge 0$ is the unknown distribution function, $\B[\cdot]$ is a bounded operator which describes aggregation dynamics and $D(\cdot)\ge0$ is a diffusion function. 

A typical example is given by mean-field models of collective behavior where the nonlocal operator $\B[\cdot]$ has the form
 \be\label{eq:B_collective}
 \B[f](w,t) = S(w) + \int_{\RR^d} P(w,w_*)(w-w_*)f(w_*,t)dw_*,
 \ee
 with $P:\RR^{d\times d }\rightarrow \RR^+$ and $S:\RR^d\rightarrow \RR^d $. With the choice \eqref{eq:B_collective} equation \eqref{eq:NAD} describes typical features of the collective behavior in multiagent systems with nonlocal type interactions. These models of collective behavior has been extensively discussed in the last decades at the particle, kinetic and hydrodynamic level \cite{APTZ,APZc,BarDeg,BCCD,CCH,CChoH,CFRT,CS,DOCBC,T}. In particular, many heterogeneous phenomena like swarming behaviors, human crowds motion and formation of wealth distributions are described by these type of PDEs under special assumptions. We refer to \cite{NPT,PT2}, and the references therein, for a recent overview of such models.
 
In the following, we focus on the construction of numerical methods for such problems which are able to preserve the structural properties of the PDE, like non negativity of the solution, entropy dissipation and large time behavior. The methods here developed are second order accurate, they do not require any restriction on the mesh size and are capable to capture the asymptotic steady states with arbitrary accuracy. These properties are essential for a correct description of the underlying physical problem.

The derivation of the schemes follows the main lines of the seminal work of Chang--Cooper for the linear Fokker-Planck equation \cite{BuetDellacherie2010,ChCo,MB,SG}.
However, in the nonlinear case, the exact stationary solution is unknown and a more advanced treatment is needed in order to find a good approximation to the problem. Similar approaches for nonlinear Fokker-Planck equations were previously derived in \cite{BCDS,LLPS}. Related methods for the case of nonlinear degenerate diffusions equations were proposed in \cite{BCF,CJS} and with nonlocal terms in \cite{BCW,CCH}. We refer also to \cite{AlbiPareschi2013ab} for the development of methods based on stochastic approximations and to \cite{Gosse} for a recent survey on schemes which preserve steady states of balance laws and related problems.

Although we derive the schemes in the case of Fokker-Planck equations, the methods can be easily applied to more general problems where the solution depends on additional parameters and the PDE is of Vlasov-Fokker-Planck type. In this case, preservation of the steady states is of fundamental importance in order to develop asymptotic-preserving methods \cite{DP15}.

The rest of the paper is organized as follows. In the next Section we first derive the Chang-Cooper type schemes in one-dimension with a particular attention to the steady state preserving properties. We then prove non negativity of solutions for explicit and semi-implicit schemes and entropy inequality for a class of one dimensional Fokker-Planck models. In Section 3 we introduce a modification of the schemes based on a more general entropy dissipation principle. We show that these entropic schemes preserve stationary solutions and derive sufficient conditions for non negativity of explicit and semi-implicit schemes. Several applications of the schemes are finally presented in Section 4 for various nonlinear Fokker-Planck problems describing collective behaviors in socio-economic and life sciences. Some conclusions are reported at the end of the manuscript. Extension to high order semi-implicit methods and the multi-dimensional case are considered in separate appendices.

\section{Chang-Cooper type schemes}\label{sec:3}
In the following we focus on the design of numerical schemes for \eqref{eq:NAD} which we rewrite in divergence form as
\be\label{eq:NAD2}
\partial_t f(w,t) = \nabla_w \cdot [(\B[f](w,t)+\nabla_w D(w))f(w,t)+D(w)\nabla_w f(w,t)].
\ee
We can define the $d$-dimensional flux 
\[
\F[f](w,t) = (\B[f](w,t)+\nabla_w D(w))f(w,t)+D(w)\nabla_w f(w,t),
\]
therefore \eqref{eq:NAD2} reads
\be\label{eq:NAD_dim}
\partial_t f(w,t) = \nabla_w \cdot \F(w,t).
\ee

\subsection{Derivation of the schemes}
In the one-dimensional case $d=1$ equation \eqref{eq:NAD_dim} becomes
\be\label{eq:FP_flux}
\partial_t f(w,t) = \partial_w \F[f](w,t),\quad w\in {\Omega} \subseteq \RR,
\ee
where now
\be
\label{eq:flux}
\F[f](w,t) = ( \B[f](w,t)+D'(w)) f(w,t)+D(w)\partial_w f(w,t)
\ee
using the compact notation $D'(w)=\partial_w D(w)$. Typically, when ${\Omega}$ is a finite size set the problem is complemented with no-flux boundary conditions at the extremal points. In the sequel we assume $D(w) \geq \alpha > 0$ in the internal points of $\Omega$.

We introduce an uniform spatial grid $w_i \in \Omega$, such that $w_{i+1}-w_i=\Delta w$. We denote as usual $w_{i\pm 1/2}=w_i\pm \Delta w/2$ and 
consider the conservative discretization
\be\label{eq:dflux}
\frac{d}{dt}f_i(t) = \dfrac{\F_{i+1/2}(t)-\F_{i-1/2}(t)}{\Delta w},
\ee
where for each $t\ge 0$, $f_i(t)$ is an approximation of $f(w_i,t)$  and $\F_{i\pm 1/2}[f](t)$ is the numerical flux function characterizing the  discretization. 

Let us set ${\C}[f](w,t)=\B[f](w,t)+D'(w)$ and adopt the notations $D_{i+1/2}=D(w_{i+1/2})$, $D'_{i+1/2}=D'(w_{i+1/2})$. We will consider a general flux function which is combination of the grid points $i+1$ and $i$ 
\be\begin{split}\label{eq:CC_flux}
\F_{i+ 1/2} = \tilde{\C}_{i+1/2}\tilde{f}_{i+1/2}+D_{i+1/2}\dfrac{f_{i+1}-f_i}{\Delta w},
\end{split}\ee
where 
\be\label{eq:f_CC}
\tilde{f}_{i+1/2}=(1-\delta_{i+1/2})f_{i+1}+\delta_{i+1/2}f_i.
\ee
Here, we aim at deriving suitable expressions for $\delta_{i+1/2}$ and $\tilde{\C}_{i+1/2}$ in such a way that the method yields nonnegative solutions, without restriction on $\Delta w$, and preserves the steady state of the system with arbitrary order of accuracy.

For example, the standard approach based on central difference is obtained by considering for all $i$ the quantities 
 \[
 \delta_{i+1/2}=1/2, \qquad \tilde{\C}_{i+1/2}={\C}[f](w_{i+1/2},t).
 \]
 It is well-known, however, that such a discretization method is subject to restrictive conditions over the mesh size $\Delta w$ in order to keep non negativity of the solution.  

First, observe that when the numerical flux vanishes from \eqref{eq:CC_flux} we get 
\[
\dfrac{f_{i+1}}{f_i} = \dfrac{-\delta_{i+1/2}\tilde{\C}_{i+1/2}+\dfrac{D_{i+1/2}}{\Delta w}}{(1-\delta_{i+1/2})\tilde{\C}_{i+1/2}+\dfrac{D_{i+1/2}}{\Delta w}}.
\]
Similarly, if we consider the analytical flux imposing ${\cal F}[f](w,t) \equiv 0$, we have
\be\label{eq:FP_steady}
D(w)\partial_w f(w,t) = -(\B[f](w,t)+D'(w))f(w,t),
\ee
which is in general not solvable, except in some special cases due to the nonlinearity on the right hand side. We may overcome this difficulty in the quasi steady-state approximation integrating equation \eqref{eq:FP_steady} on the cell $[w_i,w_{i+1}]$ 
\[
\int_{w_i}^{w_{i+1}}\dfrac{1}{f(w,t)}\partial_w f(w,t)dw = -\int_{w_i}^{w_{i+1}}\dfrac{1}{D(w)}(\B[f](w,t)+D'(w))dw,
\]
which gives
\be\label{eq:quasi_SS}
\dfrac{f(w_{i+1},t)}{f(w_i,t)} = \exp \left\{ -\int_{w_i}^{w_{i+1}}\dfrac{1}{D(w)}(\B[f](w,t)+D'(w))dw  \right\}.
\ee
The terminology quasi steady-state was introduced originally by Chang and Cooper in \cite{ChCo}, the "quasi" coming from the fact that, with general time-dependent coefficients, no time-stabilization can be expected.

Now, by equating the ratio $f_{i+1}/f_i$ and $f(w_{i+1},t)/f(w_i,t)$ of the numerical and exact flux, and setting
\be\label{eq:B_tilde}
\tilde{\C}_{i+1/2}=\dfrac{D_{i+1/2}}{\Delta w}\int_{w_i}^{w_{i+1}}\dfrac{\B[f](w,t)+D'(w)}{D(w)}dw
\ee
we recover
\be\label{eq:delta}
\delta_{i+1/2} = \dfrac{1}{\lambda_{i+1/2}}+\dfrac{1}{1-\exp(\lambda_{i+1/2})}, 
\ee
where
\be\label{eq:lambda_high}
\lambda_{i+1/2}=\int_{w_i}^{w_{i+1}}\dfrac{\B[f](w,t)+D'(w)}{D(w)}dw=\frac{\Delta w\,\tilde{\C}_{i+1/2}}{D_{i+1/2}}.
\ee
We can state the following 
\begin{proposition}
The numerical flux function \eqref{eq:CC_flux}-\eqref{eq:f_CC} with $\tilde{\C}_{i+1/2}$ and $\delta_{i+1/2}$ defined by \eqref{eq:B_tilde} and \eqref{eq:delta}-\eqref{eq:lambda_high} vanishes when the corresponding flux (\ref{eq:flux}) is equal to zero over the cell $[w_i,w_{i+1}]$. Moreover the nonlinear weight functions $\delta_{i+1/2}$ defined by \eqref{eq:delta}-\eqref{eq:lambda_high} are such that $\delta_{i+1/2} \in (0,1)$.  
\end{proposition}

The latter result follows from the simple inequality $\exp(x) \geq 1 + x$. We refer to this type of schemes as structure preserving Chang-Cooper (SP-CC) type schemes.

By discretizing \eqref{eq:lambda_high} through the midpoint rule 
\[
\int_{w_i}^{w_{i+1}}\dfrac{\B[f](w,t)+D'(w)}{D(w)}dw = \dfrac{\Delta w(\B_{i+1/2}+D'_{i+1/2})}{D_{i+1/2}}+O(\Delta w^3),
\]
we obtain the second order method defined by 
\be\label{eq:lambdamid}
\lambda_{i+1/2}^{\textrm{mid}} = \dfrac{\Delta w(\B_{i+1/2}+ D'_{i+1/2})}{D_{i+1/2}}
\ee
and 
\be
\label{eq:deltamid}
\delta_{i+1/2}^{\textrm{mid}} = \dfrac{D_{i+1/2}}{\Delta w(\B_{i+1/2}+ D'_{i+1/2})}+\dfrac{1}{1-\exp(\lambda_{i+1/2}^{\textrm{mid}})}.
\ee
Higher order accuracy of the steady state solution can be obtained using suitable higher order quadrature formulas for the integral \eqref{eq:B_tilde}. We refer to Section \ref{sec:applications} for examples and more details. For linear problems of the form $\B[f](w,t)=\B(w)$ with constant diffusion $D'=0$, the above scheme (\ref{eq:lambdamid})-(\ref{eq:deltamid}) is usually referred to as the Chang-Cooper method \cite{ChCo,MB}. In particular, if $\B(w)$ is a first order polynomial in $w$ as in \cite{BuetDellacherie2010} the midpoint rule is equivalent to the exact evaluation of the integral (\ref{eq:B_tilde}).

\begin{remark}\label{rem:grad}~
\begin{itemize}
\item 
If we consider the limit case $D_{i+1/2}\to 0$, $D'_{i+1/2}\to 0$ in \eqref{eq:lambdamid}-\eqref{eq:deltamid} we obtain the weights
\[
\delta_{i+1/2}=
\left\{
\begin{array}{cc}
 0, & \B_{i+1/2}>0,  \\
 1, & \B_{i+1/2}<0  \\
\end{array}
\right.
\]
and the scheme locally reduces to a first order upwind scheme for the corresponding continuity equation.
\item
For linear problems of the form $\B[f](w,t)=\B(w)$ the exact stationary state $f^{\infty}(w)$ can be directly computed from the solution of
\be\label{eq:FP_steady2}
D(w)\partial_w f^{\infty}(w) = -(\B(w)+D'(w))f^{\infty}(w),
\ee
together with the boundary conditions. Explicit examples of stationary states will be reported in Section \ref{sec:applications}.

Using the knowledge of the stationary state we have 
\be\label{eq:SS}
\dfrac{f^{\infty}_{i+1}}{f^{\infty}_i} = \exp \left\{ -\int_{w_i}^{w_{i+1}}\dfrac{1}{D(w)}(\B(w)+D'(w))dw  \right\}=\exp \left(-\lambda^{\infty}_{i+1/2} \right),
\ee 
therefore 
\be
\label{eq:lambda_inf}
\lambda^{\infty}_{i+1/2}=\log \left(\frac{f_{i}^{\infty}}{f_{i+1}^\infty}\right)\ee
and
\be\label{eq:delta_inf}
\delta^{\infty}_{i+1/2} = \dfrac{1}{\log(f_i^{\infty})-\log(f_{i+1}^{\infty})}+\dfrac{f_{i+1}^{\infty}}{f_{i+1}^{\infty}-f_{i}^{\infty}}. 
\ee
In this case, the numerical scheme preserves the steady state exactly. Finally, in Table \ref{tb:1} we summarize the different expressions of the weight functions. 

\end{itemize}
\end{remark}

\begin{table}[t]
\caption{Different expressions of the weights in (\ref{eq:f_CC})}
\begin{center}
\begin{tabular}{l|cc}
\hline
\hline\\[-.3cm]
{\bf Scheme} & $\delta_{i+1/2}$ & $\lambda_{i+1/2}$\\[+.2cm]
\hline
\hline\\[-.3cm]
SP-CC & $\dfrac{1}{\lambda_{i+1/2}}+\dfrac{1}{1-\exp(\lambda_{i+1/2})}$ & $\displaystyle\int_{w_i}^{w_{i+1}}\dfrac{\B[f](w,t)+D'(w)}{D(w)}dw$\\[+.4cm]
\hline\\[-.3cm]
SP-CC$_2$ (midpoint) & $\dfrac{1}{\lambda_{i+1/2}}+\dfrac{1}{1-\exp(\lambda_{i+1/2})}$ & $\displaystyle\dfrac{\Delta w(\B_{i+1/2}+ D'_{i+1/2})}{D_{i+1/2}}$\\[+.4cm]
\hline\\[-.3cm]
SP-CC$_E$ (exact) & $\dfrac{1}{\log(f_i^{\infty})-\log(f_{i+1}^{\infty})}+\dfrac{f_{i+1}^{\infty}}{f_{i+1}^{\infty}-f_{i}^{\infty}}$ & $\displaystyle\log \left(\frac{f_{i}^{\infty}}{f_{i+1}^\infty}\right)$\\[+.4cm]
\hline
\end{tabular}
\end{center}
\label{tb:1}
\end{table}%

\subsection{Main properties}
In order to study the structural properties of the numerical schemes, like conservations, non negativity and entropy property, we restrict to the one-dimensional case. To start with we consider the following simple result. 
\begin{lemma}
Let us consider the scheme \eqref{eq:dflux}-\eqref{eq:CC_flux} for $i=0,\ldots,N$ with no flux boundary conditions $\F_{N+1/2}=\F_{-1/2}=0$. We have 
\[
\sum_{i=0}^N \frac{d}{dt}f_i(t) = 0,\quad \forall\, t>0.
\]
\end{lemma}
The proof is a simple consequence of the telescopic summation property and the no flux boundary conditions.

\subsubsection{Positivity preservation}
Concerning non negativity, first we prove a result for the explicit scheme. We introduce a time discretization $t^n=n\Delta t$ with $\Delta t>0$ and $n = 0,\dots,T$ and consider the simple forward Euler method
\be\label{eq:NAD_dimd}
f^{n+1}_i=f^n_i + \Delta t \dfrac{\F_{i+1/2}^n-\F_{i-1/2}^n}{\Delta w}.
\ee  
\begin{proposition}
Under the time step restriction 
\be
\Delta t\le \dfrac{\Delta w^2}{2(M\Delta w+D)},\quad
M = \max_{i} |\tilde{\C}_{i+1/2}^n|,\quad D = \max_{i} {D}_{i+1/2}, 
\label{eq:nu}
\ee
the explicit scheme \eqref{eq:NAD_dimd} with flux defined by \eqref{eq:delta}-\eqref{eq:lambda_high} preserves nonnegativity, i.e 
$ f^{n+1}_i\ge 0$ if $f^n_i\ge 0$, $i=0,\dots,N$. 
\end{proposition}
\begin{proof}
The scheme reads
\be\begin{split}\label{eq:prop1_f}
& f_i^{n+1} = f_i^n + \dfrac{\Delta t}{\Delta w}\Bigg[ \left( (1-\delta_{i+1/2}^n)\tilde{\C}_{i+1/2}^{n}+\dfrac{D_{i+1/2}}{\Delta w} \right)f_{i+1}^n\\
& +\left(\tilde{\C}_{i+1/2}^n\delta_{i+1/2}^n-\tilde{\C}_{i-1/2}^n(1-\delta_{i-1/2}^n)-\dfrac{1}{\Delta w}(D_{i+1/2}+D_{i-1/2})\right)f_i^n \\
&- \left(\tilde{\C}^n_{i-1/2}\delta_{i-1/2}^n-\dfrac{D_{i-1/2}}{\Delta w}\right)f_{i-1}^n \Bigg].
\end{split}\ee
From \eqref{eq:prop1_f} we have a convex combination if the coefficients of $f_{i+1}^n$ and $f_{i-1}^n$  satisfy
\[
\begin{split}
(1-\delta_{i+1/2})\tilde{\C}_{i+1/2}^n+\dfrac{D_{i+1/2}}{\Delta w}\ge 0, \qquad
-\delta_{i-1/2}\tilde{\C}_{i-1/2}^n+\dfrac{D_{i-1/2}}{\Delta w}\ge 0,
\end{split}
\]
or equivalently 
\[
\begin{split}
\lambda_{i+1/2}\left(1-\dfrac{1}{1-\exp{(\lambda_{i+1/2})}}\right)\ge 0, \qquad \dfrac{\lambda_{i-1/2}}{\exp{(\lambda_{i-1/2})}-1}\ge 0,
\end{split}
\]
which holds true thanks to the properties of the exponential function. In order to ensure the non negativity of the scheme the time step should satisfy the restriction $\Delta t\le {\Delta w}/{\nu}$, with 
\[
\nu = \max_{0\le i\le N} \Big\{ -\tilde{\C}_{i+1/2}^n\delta^n_{i+1/2}+\tilde{\C}_{i-1/2}^n(1-\delta_{i-1/2}^n)+\dfrac{D_{i+1/2}+D_{i-1/2}}{\Delta w} \Big\}.
\]
Being $M$ and $D$ defined in \eqref{eq:nu}, and $0\le \delta_{i\pm 1/2}\le 1$, we obtain the prescribed bound. 
\end{proof}
\begin{remark}
Higher order SSP methods {\rm \cite{GST}} are obtained by considering a convex combination of forward Euler methods. Therefore, the non negativity result can be extended to general SSP methods. 
\end{remark}
In practical applications, it is desirable to avoid the parabolic restriction $\Delta t = O(\Delta w^2)$ of explicit schemes. Unfortunately, fully implicit methods originate a nonlinear system of equations due to the nonlinearity of $\B[f]$ and the dependence of the weights $\delta_{i\pm 1/2}$ from the solution. However, we can prove that nonnegativity of the solution holds true also for the semi-implicit case
\be
f^{n+1}_i=f^n_i + \Delta t \dfrac{\hat{\F}_{i+1/2}^{n+1}-\hat{\F}_{i-1/2}^{n+1}}{\Delta w},
\label{eq:semi}
\ee
where
\be
\hat{\F}_{i+1/2}^{n+1}= \tilde{\C}_{i+1/2}^n \left[ (1-\delta_{i+1/2}^n)f_{i+1}^{n+1}+\delta^n_{i+1/2}f_i^{n+1} \right]+D_{i+1/2}\dfrac{f_{i+1}^{n+1}-f_i^{n+1}}{\Delta w}.
\ee
We have 
\begin{proposition}\label{prop:semiimplicit_CFL}
Under the time step restriction 
\be\label{eq:time_step_implicit}
\Delta t< \dfrac{\Delta w}{2M},\qquad M = \max_{i}|\tilde{\C}^n_{i+1/2}|
\ee
the semi-implicit scheme (\ref{eq:semi}) preserves nonnegativity, i.e 
$ f^{n+1}_i\ge 0$ if $f^n_i\ge 0$, $i=0,\dots,N$.
\label{prop:semi} 
\end{proposition}
\begin{proof}
Equation \eqref{eq:semi} corresponds to
\[
\begin{split}
& f_i^{n+1}\left\{1-\dfrac{\Delta t}{\Delta w}\left[\tilde{\C}^n_{i+1/2}\delta_{i+1/2}^{n}-\tilde{\C}_{i-1/2}^{n}(1-\delta_{i-1/2}^{n})-\dfrac{1}{\Delta w}(D_{i+1/2}+D_{i-1/2})\right] \right\} \\
&+ f_{i+1}^{n+1} \left\{ -\dfrac{\Delta t}{\Delta w}\left[ (1-\delta_{i+1/2}^{n})\tilde{\C}_{i+1/2}^{n}+\dfrac{D_{i+1/2}}{\Delta w} \right] \right\}\\
&+ f_{i-1}^{n+1}\left\{ -\dfrac{\Delta t}{\Delta w}\left[-\tilde{\C}_{i-1/2}^{n}\delta^{n}_{i-1/2}+\dfrac{D_{i-1/2}}{\Delta w}\right] \right\}=f_i^n
\end{split}
\]
thanks to the definition of the flux function introduced in \eqref{eq:CC_flux}-\eqref{eq:f_CC}. Using the indentity $\lambda_{i+1/2}^{n} = \Delta w{\tilde{\C}_{i+1/2}^n}/{D_{i+1/2}}$ we obtain
\[
\begin{split}
&f_i^{n+1}\left\{1+\dfrac{\Delta t}{\Delta w^2}\left[ D_{i+1/2}\dfrac{\lambda_{i+1/2}^{n}}{\exp(\lambda_{i+1/2}^{n})-1}+D_{i-1/2}\dfrac{\lambda_{i-1/2}^{n}}{\exp(\lambda_{i-1/2}^{n})-1}\exp(\lambda_{i-1/2}^{n}) \right]\right\}\\
& +f_{i+1}^{n+1}\left\{-\dfrac{\Delta t}{\Delta w^2}D_{i+1/2}\dfrac{\lambda_{i+1/2}^{n}}{\exp(\lambda_{i+1/2}^{n})-1}\exp(\lambda_{i+1/2}^{n})\right\} \\
& + f_{i-1}^{n}\left\{ -\dfrac{\Delta t}{\Delta w^2}D_{i-1/2}\dfrac{\lambda_{i-1/2}^{n}}{\exp(\lambda_{i-1/2}^{n})-1} \right\} = f_i^n.
\end{split}
\] 
Let us denote $\alpha_{i+1/2}^{n}=\dfrac{\lambda_{i+1/2}^{n}}{\exp(\lambda_{i+1/2}^{n})-1} \geq 0$ and 
\be\begin{split}\label{eq:RQP}
& R_i^{n} = 1+\dfrac{\Delta t}{\Delta w^2}\left[ D_{i+1/2}\alpha_{i+1/2}^{n}+D_{i-1/2}\alpha_{i-1/2}^{n}\exp(\lambda_{i-1/2}^{n})  \right] \\
& Q_i^{n} = -\dfrac{\Delta t}{\Delta w^2} D_{i+1/2}\alpha_{i+1/2}^{n}\exp(\lambda_{i+1/2}^{n})\\
& P_i^{n} = -\dfrac{\Delta t}{\Delta w^2}D_{i-1/2}\alpha_{i-1/2}^{n},
\end{split}\ee
we can write    
\be\label{eq:implicit_2}
R_i^{n} f_i^{n+1} - Q_i^{n}f_{i+1}^{n+1}-P_i^{n}f_{i-1}^{n+1}=f_i^n.
\ee
If we introduce the matrix 
\be
(\mathcal A[f^{n}])_{ij}=
\begin{cases}
R_i^{n}, & j=i\\
-Q_i^{n}, & j = i+1,\quad 0\le i\le N-1 \\
-P_i^{n}, & j = i-1,\quad 1\le i \le N,
\end{cases}
\label{eq:A}
\ee
with $R_i^{n}> 0$, $Q_i^{n}\geq 0$, $P_i^{n}\geq 0$ defined in \eqref{eq:RQP} the semi-implicit scheme may be expressed in matrix form as follows
\be\label{eq:iter_S}
\mathcal A[\textbf{f}^{n}] \textbf{f}^{n+1}=\textbf{f}^n,
\ee
with $\textbf{f}^{n}=\left(f_0^n,\dots,f_N^n\right)$. Since $\textbf{f}^{n}\ge 0$, in order to prove that $\textbf{f}^{n+1}\ge 0$ it is sufficient to show $\mathcal A[f^n]^{-1}\ge 0$. Now, $\mathcal A[\cdot]$ is a tridiagonal matrix with positive diagonal elements and if $\mathcal A$ is strictly diagonally dominant we can conclude that $\mathcal A^{-1}\ge 0$.

The matrix $\mathcal A$ is strictly diagonally dominant if and only if 
\[
|R_i^n|>|Q_i^n|+|P_i^n|,\qquad i=0,1\dots,N,
\]
condition which holds true if
\[
\begin{split}
1&> \dfrac{\Delta t}{\Delta w^2}\left[ D_{i+1/2}\alpha_{i+1/2}^{n}\left(\exp(\lambda_{i+1/2}^{n})-1\right)-D_{i-1/2}\alpha_{i-1/2}^{n}\left( \exp(\lambda_{i-1/2}^{n})-1 \right) \right]\\
&=\dfrac{\Delta t}{\Delta w^2}\left[ D_{i+1/2}\lambda_{i+1/2}^{n}-D_{i-1/2}\lambda_{i-1/2}^{n} \right]
= \dfrac{\Delta t}{\Delta w}\left[\tilde{\C}_{i+1/2}^{n}-\tilde{\C}_{i-1/2}^{n} \right].
\end{split}
\]
\end{proof}
\begin{remark}~
\begin{itemize}
\item Higher order semi-implicit approximations can be constructed following {\rm \cite{BFR}}. An example of second order semi-implicit IMEX scheme is given in Appendix A. We mention also \cite{MB} where a second order semi-implicit BDF method has been considered. Note, however, that the determination of nonnegative semi-implicit schemes with optimal stability regions is an open problem which goes beyond the purpose of the present manuscript. 
 
\item 
Although fully implicit schemes originate a nonlinear system of equations, we remark that the same argument used in Proposition \ref{prop:semi} permits to prove nonnegativity of the scheme even with the fully implicit fluxes
\be
\F^{n+1}_{i+1/2}= \tilde{\C}_{i+1/2}^{n+1} \left[ (1-\delta_{i+1/2}^{n+1})f_{i+1}^{n+1}+\delta_{i+1/2}f_i^{n+1} \right]+D_{i+1/2}\dfrac{f_{i+1}^{n+1}-f_i^{n+1}}{\Delta w},
\ee
with 
\be
\label{eq:time_step_implicit2}
\Delta t < \dfrac{\Delta w}{2M}, \quad M = \max_{0\le i\le N}|\tilde{\C}^{n+1}_{i+1/2}|.
\ee
In fact, we obtain the nonlinear system 
\be
\mathcal{A}[{\bf f}^{n+1}]{\bf f}^{n+1}={\bf{f}}^n,
\ee
where the matrix $\mathcal A[{\bf f}^{n+1}]$ has the same structure \eqref{eq:A} with the entries evaluated at time $n+1$. The above system can  be solved iteratively at each time step
\be\begin{split}
{\bf{f}}^{n+1}_0 &= {\bf{f}}^n,\\
{\bf{f}}^{n+1}_{k+1} &= \mathcal A^{-1}[{\bf{f}}_k^{n+1}]{\bf{f}}^n, \qquad k = 0,1,\dots
\end{split}\ee
where now each iteration step is explicit and can be made non negative under a stability restriction analogous to \eqref{eq:time_step_implicit}. Therefore, if ${\bf{f}}_k^{n+1}\rightarrow {\bf{f}}^{n+1}$ as ${k\rightarrow +\infty}$ we can infer the nonnegativity of the scheme under the condition \eqref{eq:time_step_implicit2}, being $\mathcal A[{\bf{f}}^{n+1}]\ge 0$ strictly diagonally dominant and then $\mathcal A[{\bf{f}}^{n+1}]^{-1}\ge 0$. In general, the convergence properties of the above iterative method depend on the nonlinear flux function $\mathcal{B}[f]$ which defines the coefficients in $\mathcal A[{\bf{f}}]$. For example, since by nonnegativity and mass conservation $\|f_{k+1}^{n+1}\|_1=\|f_{k}^{n+1}\|_1$, $k \geq 0$, where $\|\cdot\|_1$ denotes the discrete $L_1$ norm, convergence is guaranteed for any choice of the initial value $f^{n+1}_0$ if $\mathcal A^{-1}[{\bf{\cdot}}]{\bf{f}}^n$ is a contraction \cite{CTK}.

\end{itemize}
\end{remark}


\subsubsection{Relative entropy dissipation}
In order to discuss the entropy property we consider the prototype equation 
\be\label{eq:wu}
\partial_t f(w,t) = \partial_w \left[ (w-u)f(w,t) + \partial_w (D(w)f(w,t)) \right], \qquad w\in  I=[-1,1],
\ee
with $-1<u<1$ a given constant and boundary conditions
\be\label{eq:wu_boundary}
\partial_w (D(w)f(w,t))+(w-u)f(w,t) = 0, \qquad w=\pm1.
\ee
If the stationary state $f^\infty$ exists equation \eqref{eq:wu} may be written in the \emph{Landau form} as
\be
\partial_t f(w,t) = \partial_w \left[ D(w)f(w,t)\partial_w \log\left( \dfrac{f(w,t)}{f^{\infty}(w)} \right) \right],
\label{eq:landau}
\ee
or in the \emph{non logarithmic Landau form} as
\be\label{eq:nonlog_landau}
\partial_t f(w,t) = \partial_w \left[ D(w)f^{\infty}(w)\partial_w \left(\dfrac{f(w,t)}{f^{\infty}(w)}\right) \right].
\ee
We define the relative entropy for all positive functions $f(w,t),g(w,t)$ as follows
\be\label{eq:rel_entropy}
\mathcal H(f,g) = \int_I f(w,t) \log\left(\dfrac{f(w,t)}{g(w,t)} \right),
\ee
we have \cite{FPTT_K}
\be
\dfrac{d}{dt}\mathcal H(f,f^{\infty}) = -\mathcal I_D(f,f^{\infty}),
\ee
where the dissipation functional $\mathcal I_D(\cdot,\cdot)$ is defined as 
\be\begin{split}
\mathcal I_D(f,f^{\infty})& = 
 \int_{\mathcal I} D(w)f(w,t)\left(\partial_w \log \left( \dfrac{f(w,t)}{f^{\infty}(w)}\right)\right)^2dw,\\
& = 
 \int_{\mathcal I} D(w)f^{\infty}(w,t)\partial_w \log \left( \dfrac{f(w,t)}{f^{\infty}(w)}\right)\partial_w\left(\frac{f}{f^{\infty}}\right)dw.\\ 
\end{split}\ee
Of course we might consider other entropies like the $L^2$ entropy which is defined as
\be\begin{split}
&L^2(f,f^{\infty}) = \int_{I}\dfrac{(f(w,t)-f^{\infty}(w))^2}{f^{\infty}(w)}dw,\\
&\dfrac{d}{dt}L_2(f,f^{\infty})= -J_D(f,f^{\infty}),
\end{split}\ee
with 
\be
J_D(f,f^{\infty}) = 2\int_{\mathcal I}D(w)f^{\infty}\left(\partial_w \left(\dfrac{f(w,t)}{f^{\infty}(w)}\right)^2\right),
\ee
see \cite{FPTT_K} for further examples.

Note that, since the definition of relative entropy implies the existence of a steady state $f^{\infty}$ then the considerations of the second part of Remark \ref{rem:grad} apply. Therefore, the weights in the Chang-Cooper type method can be evaluated exactly and are given by (\ref{eq:lambda_inf})-(\ref{eq:delta_inf}). This property is used to prove the following results.

\begin{lemma}\label{lem:flux_CCE}
 In the case $\B[f](w,t)=\B(w)$ the numerical flux function (\ref{eq:CC_flux})-(\ref{eq:f_CC}) with $\tilde{\C}_{i+1/2}$ and $\delta_{i+1/2}$ given by (\ref{eq:B_tilde})-(\ref{eq:delta}) can be written in the form (\ref{eq:nonlog_landau}) and reads
 \be
 \F_{i+1/2}^n = \dfrac{D_{i+1/2}}{\Delta w} \hat f_{i+1/2}^{\infty} \left( \dfrac{f^n_{i+1}}{f^{\infty}_{i+1}}-\dfrac{f^n_i}{f^{\infty}_i} \right),
 \label{eq:nnll}
 \ee
 with 
 \[
 \hat{f}^{\infty}_{i+1/2} = \dfrac{f_{i+1}^{\infty}f_i^{\infty}}{f_{i+1}^{\infty}-f_i^{\infty}}\log \left(\dfrac{f_{i+1}^{\infty}}{f_i^{\infty}}\right).
 \]
 \end{lemma}
 \begin{proof}
 In the hypothesis $\B[f](w,t)=\mathcal  B(w)$ the definition of $\lambda_{i+1/2}$ does not depends on time, i.e. $\lambda_{i+1/2}=\lambda_{i+1/2}^{\infty}$ and if a steady state exists we may write
 \[
 \log f_{i}^{\infty}-\log f_{i+1}^{\infty} = \lambda_{i+1/2}.
 \]
Furthermore, the flux function $\F^n_{i+1/2}$ assumes the following form
 \be\begin{split}
 \F^n_{i+1/2} &= \dfrac{D_{i+1/2}}{\Delta w}\left[ \lambda_{i+1/2}\tilde{f}^{n}_{i+1/2} +(f_{i+1}^n-f_i^n)\right]\\
 &=  \dfrac{D_{i+1/2}}{\Delta w}\left[ \lambda_{i+1/2} (f_{i+1}^n+\delta_{i+1/2}(f_i^n-f_{i+1}^n)) +(f_{i+1}^n-f_i^n)\right],
 \end{split}\ee
 where 
 \be
 \delta_{i+1/2} = \dfrac{1}{\log f_{i}^{\infty}-\log f_{i+1}^{\infty}}+\dfrac{f_{i+1}^{\infty}}{f_{i+1}^{\infty}-f_i^{\infty}}.
 \ee
 Hence we have
   \be\begin{split}
 \mathcal F_{i+1/2}^n = \dfrac{D_{i+1/2}}{\Delta w} \log \left( \dfrac{f^{\infty}_i}{f_{i+1}^{\infty}}\right)  &\left[f_{i+1}+ \left(\dfrac{f_i-f_{i+1}}{\log f_{i}^{\infty}-\log f_{i+1}^{\infty}}+\dfrac{f_{i+1}^{\infty}(f_i-f_{i+1})}{f_{i+1}^{\infty}-f_i^{\infty}}\right)\right.\\
 &\left.\quad+\dfrac{f_{i+1}-f_{i}}{\log f_i^{\infty}-\log f_{i+1}^{\infty}}  \right], \\
 = \dfrac{D_{i+1/2}}{\Delta w} \log\left( \dfrac{f^{\infty}_i}{f_{i+1}^{\infty}}\right)& \left( \dfrac{f_{i+1}^{\infty}f_i-f_{i}^{\infty}f_{i+1}}{f_{i+1}^{\infty}-f_i^{\infty}} \right)
 \end{split}\ee
  which gives (\ref{eq:nnll}).
 \end{proof}

\begin{theorem}\label{th:1}
Let us consider $\B[f](w,t)=w-u$ as in equation \eqref{eq:wu}. The numerical flux (\ref{eq:CC_flux})-(\ref{eq:f_CC}) with $\tilde{\C}_{i+1/2}$ and $\delta_{i+1/2}$ given by (\ref{eq:B_tilde})-(\ref{eq:delta}) satisfies the discrete entropy dissipation
\be
\dfrac{d}{dt}\mathcal H_{\Delta}(f,f^{\infty})=-\mathcal I_{\Delta}(f,f^{\infty}),
\ee
where
\be\label{eq:relative_entropy}
\mathcal H_{\Delta w}(f,f^{\infty}) = \Delta w \sum_{i=0}^N f_i \log \left(\dfrac{f_i}{f_i^{\infty}} \right)
\ee
and $I_{\Delta}$ is the positive discrete dissipation function 
\be
\mathcal I_{\Delta}(f,f^{\infty}) = \sum_{i=0}^N
 \left[ \log \left(\dfrac{f_{i+1}}{f^{\infty}_{i+1}}\right)-\log\left(\dfrac{f_i}{f_i^{\infty}}\right) \right]\cdot \left(\dfrac{f_{i+1}}{f_{i+1}^{\infty}}-\dfrac{f_i}{f_{i}^{\infty}}\right)\hat{f}_{i+1/2}^{\infty}D_{i+1/2}\ge 0.
\ee
\end{theorem}
\begin{proof}
From the definition of relative entropy we have
\[
\begin{split}
\dfrac{d}{dt} \mathcal H(f,f^{\infty})&= \Delta w\sum_{i=0}^N  \dfrac{df_i}{dt}\left(\log\left(\dfrac{f_i}{f_i^{\infty}}\right)+1\right)\\
&= \Delta w \sum_{i=0}^N \left( \log\left(\dfrac{f_i}{f_i^{\infty}}\right)+1 \right)(\F_{i+1/2}-\F_{i-1/2}),
\end{split}
\]
and after summation by parts we get
\be
\dfrac{d}{dt} \mathcal H(f,f^{\infty})=-\Delta w \sum_{i=0}^N \left[ \log \left(\dfrac{f_{i+1}}{f_{i+1}^{\infty}}\right)-\log\left(\dfrac{f_i}{f_i^{\infty}}\right) \right]\F_{i+1/2}.
\ee
Thanks to the identity of Lemma \ref{lem:flux_CCE}  we may conclude since the function $(x-y)\log(x/y)$ is non-negative for all $x,y\ge 0$.
\end{proof}

\section{Entropic schemes for gradient flow problems}
In this section we introduce a second class of structure preserving numerical scheme based on the entropy dissipation principle.
To this aim, let us consider the general class of nonlinear Fokker-Planck equation with 
gradient flow structure \cite{BCCD,CCH,CMV}
\be\label{eq:gradient_2}
\partial_t f(w,t) = \nabla_w \cdot [f(w,t)\nabla_w\xi(w,t)], \qquad w\in\Omega\subseteq\RR^d,
\ee
and no-flux boundary conditions. In the case of equation \eqref{eq:NAD} with constant diffusion $D$ 
we have
\be
\nabla_w \xi(w,t) = {\B}[f](w,t) + D {\nabla_w \log f(w,t)}.
\label{eq:xi}
\ee
We focus on the following prototype of function $\xi(w,t)$, $w\in\RR^d$ 
\be\label{eq:xi_2}
\xi = (U*f)(w,t)+D\log f(w,t),
\ee
which in our case corresponds to
\[
{\B}[f] (w,t)=\nabla_w (U*f)(w,t),
\]
with $U=U(|w|)$ an interaction potential.

The corresponding free energy is given by
\be
\mathcal E(t) = \dfrac{1}{2}\int_{\RR^d}(U*f)(w,t)f(w,t)dw+D\int_{\RR^d}\log f(w,t) f(w,t)dw.
\label{eq:Ent}
\ee 
Using the fact that $U$ is an even function we can write 
\be
\begin{split}
\dfrac{d}{dt}\mathcal E(t) &= \int_{\RR^d} \partial_t f(w,t)dw+ \int_{\RR^d} ((U*f)(w,t)+D\log f(w,t))\partial_t f(w,t)dw\\
&= \int_{\RR^d} \nabla_w \cdot [f(w,t)\nabla_w\xi(w,t)](1+\xi)dw. 
\end{split}
\ee
Hence, upon integration by parts we obtain 
the dissipation of the free energy $\mathcal E(t)$ along solutions 
\be\label{eq:dissipation}
\dfrac{d}{dt}\mathcal E(t) = -\int_{\RR^d}|\nabla_w\xi|^2 f(w,t)dw =- \mathcal I(t),
\ee
where $\mathcal I(\cdot)$ is the entropy dissipation function.
\medskip

\subsection{Derivation of the schemes}
\label{sec:ea}
In the one-dimensional case $d=1$ equation \eqref{eq:gradient_2} reads 
\be
\label{eq:gradient_2b}
\partial_t f(w,t) = \partial_w [f(w,t)\partial_w\xi(w,t)], \qquad w\in\Omega\subseteq\RR,
\ee
where 
\be
\partial_w \xi(w,t) = {\B}[f](w,t) + D {\partial_w \log f(w,t)},
\label{eq:xib}
\ee
and we assume ${\B}[f](w,t)=\partial_w (U \ast f)(w,t)$.

In order to derive schemes which satisfy the entropy dissipation property \eqref{eq:dissipation} we consider the semi-discrete version of the entropy (\ref{eq:Ent})
\be\begin{split}\label{eq:discrete_entropy}
\mathcal E_{\Delta}(t) &= \Delta w \sum_{j=0}^N \left[ \dfrac{1}{2}\Delta w\sum_{i=0}^N U_{j-i}f_if_j+D f_j\log f_j \right].
\end{split}
\ee
Therefore, we have
\[
\begin{split}
\dfrac{d}{dt}\mathcal{E}_{\Delta} &= \Delta w \sum_{j=0}^N \left[\Delta w \sum_{i=0}^N U_{j-i}f_i\dfrac{df_j}{dt} +D\left(\log f_j+1\right)\dfrac{df_j}{dt} \right]\\
&= \Delta w\sum_{j=0}^N \left[\Delta w \sum_{i=0}^NU_{j-i}f_i+D\log f_j+1\right]\dfrac{df_j}{dt}.
\end{split}
\]
Now using the general discrete conservative formulation (\ref{eq:dflux}) and the fact that $\xi_i = U*f_i+D\log f_i$ we get
\[
\dfrac{d}{dt}\mathcal{E}_{\Delta} = \sum_{j=0}^N (\xi_j+1)(\F_{j+1/2}-\F_{j-1/2}).
\]
Furthermore, after summation by parts we can write the last term as follows
\be\begin{split}\label{eq:back_E}
\dfrac{d}{dt}\mathcal{E}_{\Delta} &= -\sum_{j=0}^N(\xi_{j+1}-\xi_{j})\F_{j+1/2}.
\end{split}\ee
Now, integrating \eqref{eq:xi} over the cell $[w_i,w_{i+1}]$ we obtain
\be
\xi_{i+1}-\xi_i  = {\Delta w\,\tilde{ \B}_{i+1/2}} + D\log\left(\dfrac{f_{i+1}}{f_i}\right),
\label{eq:xj1xj}
\ee
where 
\[
\tilde{\B}_{i+1/2}=\frac1{\Delta w}\int_{w_i}^{w_{i+1}}\B[f](w,t)\,dw.
\]
Let us now consider a general scheme in the Chang-Cooper form \eqref{eq:CC_flux}, which in our case can be rewritten as
\be
\F_{i+1/2} = \left({\tilde{ \B}_{j+1/2}}+\frac{D}{\Delta w}\log\left(\dfrac{f_{j+1}}{f_j}\right)K_{j+1/2}\right)\tilde f_{j+1/2}
\ee
with 
\be
K_{j+1/2} = \frac1{\tilde f_{j+1/2}}\dfrac{f_{j+1}-f_j}{(\log f_{j+1}-\log f_j)},\qquad f_{j+1}\ne f_j.
\ee
Therefore, we have
\be
\label{eq:diss_entropy_general}
\begin{split}
\dfrac{d}{dt}\mathcal{E}_{\Delta} =-\Delta w \sum_{j=0}^N  &\left({\tilde{ \B}_{j+1/2}}+ \dfrac{D}{\Delta w}\log\left(\dfrac{f_{j+1}}{f_j}\right)\right)\\
& \left({\tilde{ \B}_{j+1/2}} + \dfrac{D}{\Delta w}\log\left(\dfrac{f_{j+1}}{f_j}\right)K_{j+1/2}\right)\tilde f_{j+1/2}.
\end{split}
\ee
Thus we cannot prove that the discrete entropy functional \eqref{eq:discrete_entropy} is dissipated by the Chang-Cooper type scheme developed in the previous sections, unless $K_{j+1/2}\equiv 1$. This latter requirement is satisfied if we consider the new entropic flux function
\be\label{eq:entropic_average}
 \tilde f^{E}_{i+1/2} =
 \begin{cases}
  \dfrac{f_{i+1}-f_i}{\log f_{i+1}-\log f_i} & f_{i+1}\ne f_i,\\
 f_{i+1} & f_{i+1}=f_i.
 \end{cases}\ee
 We will refer to the above approximation of the solution at the grid point $i+1/2$ as \emph{entropic average} of the grid points $i$ and $i+1$. In the general case of the flux function \eqref{eq:flux} with non constant diffusion the resulting numerical flux reads
\be\label{eq:F_E}
\F^E_{i+1/2} = D_{i+1/2}\left( \frac{\tilde {\C}_{i+1/2}}{D_{i+1/2}}+\dfrac{\log f_{i+1}-\log f_i}{\Delta w}  \right) \tilde f^E_{i+1/2}.
 \ee
Finally, concerning the stationary state, we obtain immediately imposing the numerical flux equal to zero
\[
\frac{\tilde {\C}_{i+1/2}}{D_{i+1/2}}+\dfrac{\log f_{i+1}-\log f_i}{\Delta w}=0,
\]  
and therefore we get
\[
\frac{f_{i+1}}{f_i} = \exp\left(-\frac{\Delta w\,\tilde {\C}_{i+1/2}}{D_{i+1/2}}\right).
\] 
By equating the above ratio with the quasi-stationary approximation \eqref{eq:quasi_SS} we get the same expression for $\tilde{\C}_{i+1/2}$ as in \eqref{eq:B_tilde}
\be\label{eq:B_tilde2}
\tilde{\C}_{i+1/2}=\dfrac{D_{i+1/2}}{\Delta w}\int_{w_i}^{w_{i+1}}\dfrac{\B[f](w,t)+D'(w)}{D(w)}dw.
\ee 
We can state the following 
\begin{proposition}
The numerical flux function \eqref{eq:F_E} with $\tilde{\C}_{i+1/2}$ defined by \eqref{eq:B_tilde2}  vanishes when the corresponding flux \eqref{eq:xib} is equal to zero over the cell $[w_i,w_{i+1}]$.  
\end{proposition}

\subsection{Main properties}
 
A fundamental result concerning the entropic average \eqref{eq:entropic_average} is the following Lemma.

 \begin{lemma}
The entropy average defined in \eqref{eq:entropic_average} may be written as a convex combination with nonlinear weights
\be
\tilde{f}^{E}_{i+1/2} = \delta_{i+1/2}^{E} f_i +(1-\delta_{i+1/2}^{E})f_{i+1},
\ee
where 
\be\label{eq:deltaE}
\delta_{i+1/2}^{E} = \dfrac{f_{i+1}}{f_{i+1}-f_i}+\dfrac{1}{\log f_i-\log f_{i+1}}\in(0,1).
\ee
 \end{lemma}
 
 \begin{proof}
 From \eqref{eq:deltaE} we have
 \[
 \begin{split}
 \tilde{f}^{E}_{i+1/2} &= f_{i+1}+\delta_{i+1/2}^{E} (f_i-f_{i+1}) \\
 &= f_{i+1}-f_{i+1}+\dfrac{f_i-f_{i+1}}{\log f_i-\log f_{i+1}} \\
 &= \dfrac{f_{i+1}-f_i}{\log f_{i+1}-\log f_i},
 \end{split}
 \]
 that is \eqref{eq:F_E}. It is a easy computation to verify that $\delta^{E}_{i+1/2}$ lies in the interval $(0,1)$.
 \end{proof}
 
 \begin{remark} As a consequence the Chang-Cooper type average \eqref{eq:f_CC} and the entropic average \eqref{eq:entropic_average} define the same quantity at the steady state when $f_i=f_i^{\infty}$. In fact, in this case \eqref{eq:deltaE} are the same as \eqref{eq:delta_inf}.
 \end{remark}
 
We can summarize our findings of Section \ref{sec:ea} as follows.
\begin{theorem}\label{pr:2}
The numerical flux \eqref{eq:F_E}-\eqref{eq:entropic_average} for a constant diffusion $D$ satisfies the discrete entropy dissipation
\be
\dfrac{d}{dt}\mathcal{E}_{\Delta}=- \mathcal I_{\Delta}(t),
\ee
where $\mathcal{E}_{\Delta}$ is given by \eqref{eq:discrete_entropy} 
and $I_{\Delta}$ is the discrete entropy dissipation function 
\be
\mathcal I_{\Delta} = \Delta w \sum_{j=0}^N  (\xi_{j+1}-\xi_j)^2  \tilde{f}^{E}_{i+1/2} \ge 0,
\ee
with $\xi_{j+1}-\xi_j$ defined as in \eqref{eq:xj1xj}.
\end{theorem}

\begin{remark}
On the contrary to the Chang-Cooper average the restrictions for the non negativity property of the solution are stronger. In fact, by the same arguments we used in the previous section, non negativity of the explicit scheme requires 
\be\begin{split}
(1-\delta^E_{i+1/2})\tilde{\C}_{i+1/2}^n+\dfrac{D_{i+1/2}}{\Delta w}\ge 0, \qquad
-\delta^E_{i-1/2}\tilde{\C}_{i-1/2}^n+\dfrac{D_{i-1/2}}{\Delta w}\ge 0.
\end{split}\ee
However, the weights do not possess any special structure that permits to avoid a constraint of the mesh size $\Delta w$ which now must satisfy
\be
\begin{split}
\Delta w \leq 
\min_i\left\{\dfrac{D_{i+1/2}}{|\tilde{\C}_{i+1/2}^n|},\dfrac{D_{i-1/2}}{|\tilde{\C}_{i-1/2}^n|}\right\}.
\end{split}
\ee
Therefore, similar to central differences, we have a restriction on the mesh size which becomes prohibitive for small values of the diffusion function $D(w)$. It is easy to verify that the same condition is necessary also for the non negativity of semi-implicit approximations.
\end{remark}

\subsubsection{The case ${\B}[f](w,t)=B(w)$}

Next we consider the case of linear flux ${\B}[f](w,t)=B(w)$.
The following Lemma holds true.

\begin{lemma}\label{prop:1E}
In the case $\B[f](w,t)=B(w)$ the numerical flux \eqref{eq:F_E}-\eqref{eq:entropic_average} corresponds to the form \eqref{eq:landau} and reads
\be\begin{split}
\tilde{\F}^E_{i+1/2}=&\dfrac{D_{i+1/2}}{\Delta w} \tilde{f}^E_{i+1/2}  \left(\log \left(\dfrac{f_{i+1}}{f^{\infty}_{i+1}}\right)-\log \left(\dfrac{f_{i}}{f^{\infty}_{i}}\right)\right).
\end{split}\ee
\end{lemma}

\begin{proof}
If a stationary $f^{\infty}(w)$ state exists it nullify the flux and we have
\[
\tilde{\C}_{i+1/2} = -\dfrac{D_{i+1/2}}{\Delta w} \left( \log f^{\infty}_{i+1}-\log f^{\infty}_i\right).
\]
From the definition of the entropic flux \eqref{eq:F_E} we obtain
\[
\begin{split}
\tilde{\F}^E_{i+1/2}=& \tilde{\C}_{i+1/2}\tilde f^E_{i+1/2}+\dfrac{D_{i+1/2}}{\Delta w}\log \dfrac{f_{i+1}}{f_i}\tilde f^E_{i+1/2}\\
=& \dfrac{D_{i+1/2}}{\Delta w} \tilde{f}^E_{i+1/2} \left[ (\log f_{i+1}-\log f_i)-(\log f^{\infty}_{i+1}-\log f^{\infty}_i) \right],
\end{split}
\]
from which we conclude.
\end{proof}
We can now state the following entropy dissipation results for problem \eqref{eq:wu} in the nonlogarithmic Landau form \eqref{eq:nonlog_landau}.
\begin{theorem}\label{th:1b}
Let us consider $\B[f](w,t)=w-u$ as in equation \eqref{eq:wu}. The numerical flux \eqref{eq:F_E}-\eqref{eq:entropic_average} with $\tilde{\C}_{i+1/2}$ given by (\ref{eq:B_tilde}) satisfies the discrete entropy dissipation
\be
\dfrac{d}{dt} \mathcal H_{\Delta}(f,f^{\infty})  = -\mathcal I^E_{\Delta}(f,f^{\infty}),
\ee
where $\mathcal H_{\Delta w}(f,f^{\infty})$ is given by \eqref{eq:relative_entropy} and $I^E_{\Delta}$ is the positive discrete dissipation function 
\be
\mathcal I^E_{\Delta}(f,f^{\infty})= \sum_{i=0}^N \left[ \log \left(\dfrac{f_{i+1}}{f^{\infty}_{i+1}}\right) -\log \left(\dfrac{f_i}{f^{\infty}_i}\right) \right]^2 D_{i+1/2} \tilde f^E_{i+1/2}\ge 0.
\ee
\end{theorem}
\begin{proof}
From
\[
\dfrac{d}{dt} \mathcal H(f,f^{\infty}) = -\Delta w \sum_{i=0}^N \left[ \log \left(\dfrac{f_{i+1}}{f^{\infty}_{i+1}}\right) -\log \left(\dfrac{f_i}{f^{\infty}_i}\right) \right] \F^E_{i+1/2}
\]
and being 
\[
\F^E_{i+1/2} = \dfrac{D_{i+1/2}}{\Delta w}
\left[\log \left(\dfrac{f_{i+1}}{f_{i+1}^{\infty}}\right)-\log \left(\dfrac{f_i}{f_i^{\infty}}\right) \right]
\tilde{f}^{E}_{i+1/2}
\]
we have
\[
\dfrac{d}{dt} \mathcal H(f,f^{\infty}) = - \sum_{i=0}^N \left[ \log \left(\dfrac{f_{i+1}}{f^{\infty}_{i+1}}\right) -\log \left(\dfrac{f_i}{f^{\infty}_i}\right) \right]^2 D_{i+1/2} \tilde f^E_{i+1/2}.
\]
\end{proof} 

\section{Applications}
\label{sec:applications}
In this section we present several numerical examples of Fokker-Planck equations solved with the structure-preserving schemes here introduced. An essential aspect for the accurate description of the steady state is the approximation of the integral defining the quasi-stationary solution
\begin{equation}\label{eq:app_lambda}
\lambda_{i+1/2} = \int_{w_i}^{w_{i+1}}\dfrac{\B[f](w,t)+D'(w)}{D(w)}dw.
\end{equation}
Except for simple linear cases, a suitable quadrature formula is required. In the following numerical examples we consider open Newton-Cotes formulas up to order 6 and Gauss-Legendre quadrature with $6$ points. 
In the sequel, we will adopt the notation SP--CC$_k$ and SP--EA$_k$, $k=2,4,6,G$ to denote the structure preserving schemes with Chang-Cooper (CC) and entropic average (EA) flux and when \eqref{eq:app_lambda} is approximated with second, fourth, sixth order Newton--Cotes quadrature or Gaussian quadrature, respectively.
Singularities at the boundaries in the integration of (\ref{eq:app_lambda}) can be avoided using open Newton--Cotes rules.

\subsection{Example 1: Opinion dynamics in bounded domains}
Let us consider the evolution of a distribution function described by \eqref{eq:NAD}, with $w\in I$, where $I=[-1,1]$, and
\be\label{eq:opinion_BD}
\B[f](w,t) = \int_I P(w,w_*)(w-w_*)f(w_*,t)dw_*, \qquad D(w)=\dfrac{\sigma^2}{2}(1-w^2)^2.
\ee
The model describes the evolution of the distribution functions of agents having opinion $w$ at time $t$ (see \cite{PT2,T} for more details).

In the simplified case $P\equiv 1$ the corresponding stationary distribution reads
\be\label{eq:opinion_stat}
f_{\infty}(w) = \dfrac{C}{(1-w^2)^2}\left(\dfrac{1+w}{1-w}\right)^{u/(2\sigma^2)}\exp\Big\{ -\dfrac{(1-uw)}{\sigma^2(1-w^2)} \Big\},
\ee
with $\sigma\in\RR$ a given parameter, $C>0$ is a normalization constant and $u=\int_I wf(w,t)dw,$. 

We consider as initial distribution
\be\label{eq:opinion_initial}
f(w,0) =\beta \left[ \exp\left(-c(w+1/2)^2\right)+\exp\left(-c(w-1/2)^2\right) \right], \qquad c = 30,
\ee
with $\beta>0$ a normalization constant. Since diffusion vanishes at the boundaries we present results for the Chang-Cooper type numerical schemes SP--CC only. 

In Figure \ref{fig:example1_L1ent} we compute the relative $L^1$ error of the numerical solution with respect to the exact \eqref{eq:opinion_stat} stationary state using $N=80$ points for the SP--CC scheme with various quadrature rules. 
It is possible to observe how the different integration methods capture the steady state with different accuracy. In particular low order quadrature rules achieve the numerical steady state faster than high order quadratures, with the Gaussian quadrature that essentially reach machine precision. In the same figure we illustrate how SP--CC scheme dissipates the relative entropy \eqref{eq:relative_entropy} in the case of two coarse grids with $N=10$ and $N=20$ points.

\begin{figure}
\centering
\includegraphics[scale=0.295]{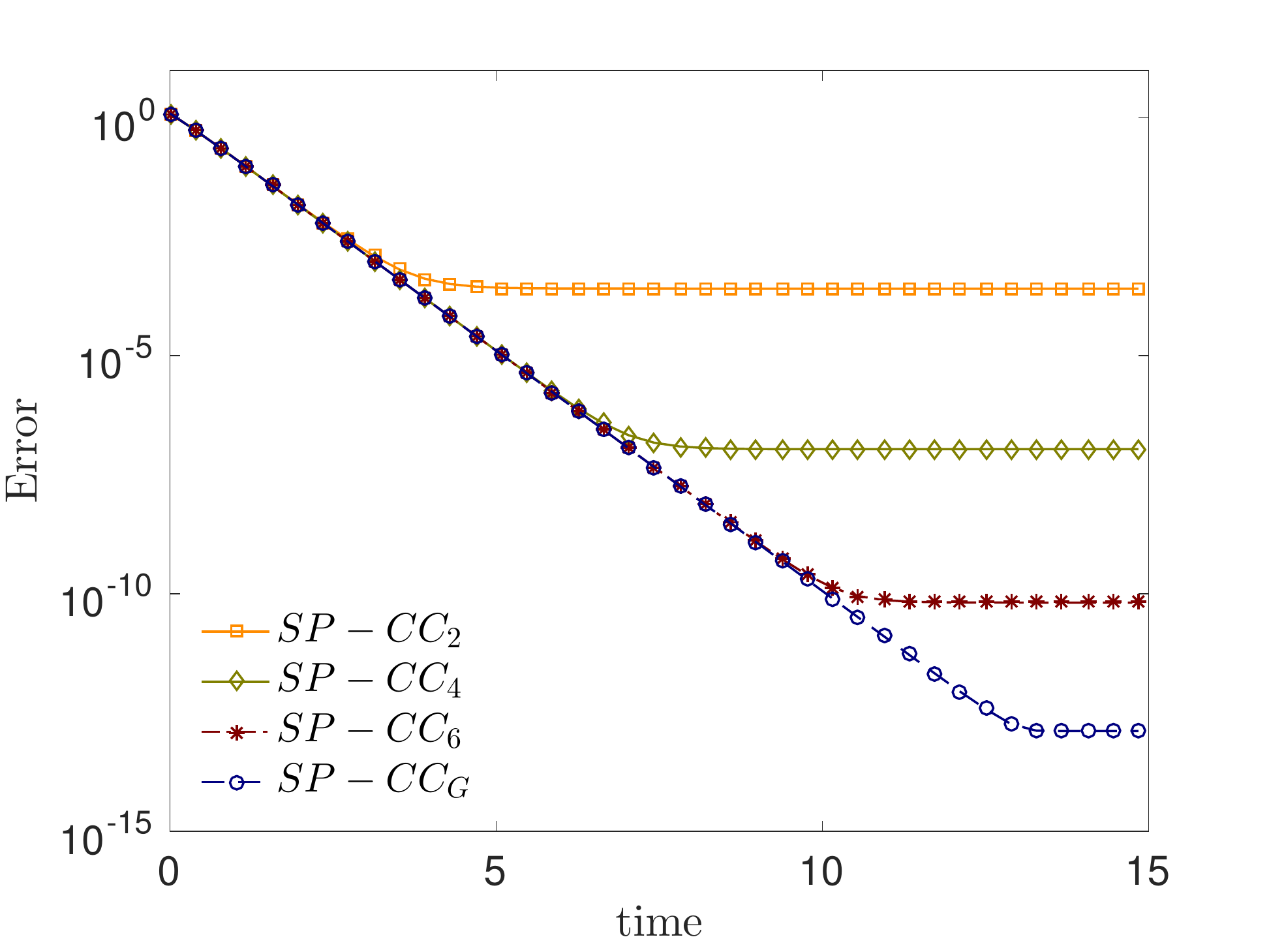}
\includegraphics[scale=0.295]{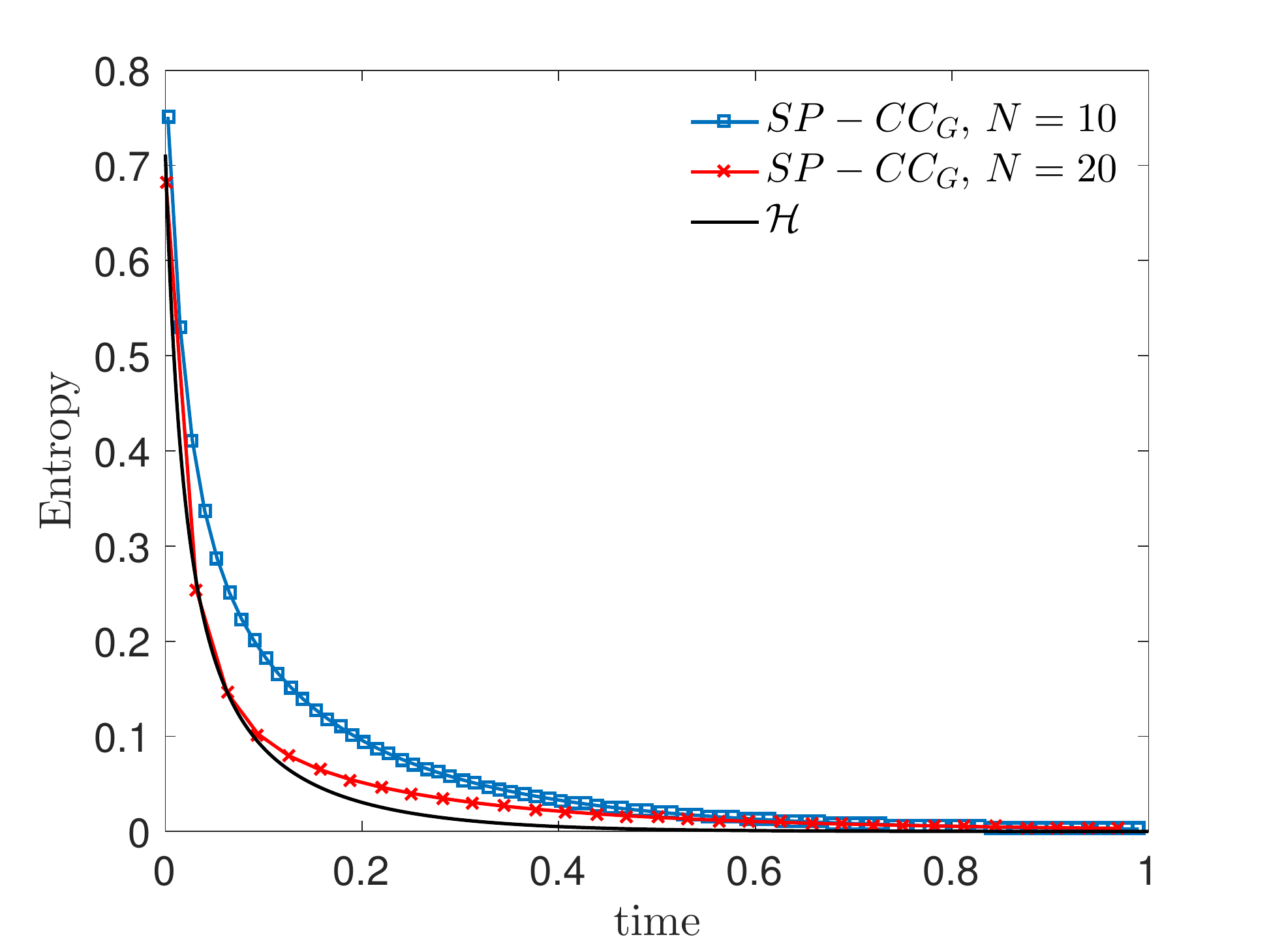}
\caption{{Example 1}. Left: evolution of the relative $L^1$ error with respect to the stationary solution \eqref{eq:opinion_stat} for the SP--CC scheme with different quadrature methods. Solution for the initial data \eqref{eq:opinion_initial} over the time interval $[0,10]$, $\sigma^2/2=0.1$, $N=80$, $\Delta t = \Delta w^2/2\sigma^2$. Right: dissipation of the numerical entropy for SP--CC scheme with Gaussian quadrature for two coarse grids with $N=10$ and $N=20$ points. }\label{fig:example1_L1ent}
\end{figure}

In Table \ref{tab:opinion} we estimate the overall order of convergence of explicit SP--CC schemes for several integration methods. Here we used $N=40,80$, with reference solutions computed with $N=640$ points. For comparison, the time integration has been performed with the explicit Euler and RK4 methods and the time step chosen in such a way that the CFL condition for the positivity of the scheme is satisfied, i.e. $\Delta t= O(\Delta w^2)$. As expected the two methods are essentially equivalent and both schemes are second order accurate in the transient regimes and assume the order of the quadrature method close to the steady state. 

In Table \ref{tab:opinion_SI}  we estimate the order of convergence of SP--CC schemes with first and second order semi--implicit methods (see Appendix \ref{appendix:A} for a detailed description of the methods). In this case the CFL condition is $\Delta t= O(\Delta w)$ and a second order method is necessary to achieve second order accuracy in the transient regime, whereas both schemes achieve higher order accordingly to the quadrature used for large times.

\begin{table}
\begin{center}
\begin{tabular}{ c || c c c c  | c c c c }
\hline
 & \multicolumn{4}{c} {$SP-CC_k$}  &  \multicolumn{4}{c} {$SP-CC_k$} \\ \hline


             Time         &     2   &  4  &  6  & G &     2   &  4  &  6  & G    \\ 
             \hline

\multirow{1}{*} {1} & 1.9456 & 1.9751  & 1.9740 & 1.9740
						& 1.9470 &  1.9773 & 1.9762 & 1.9762 \\
                             \hline
\multirow{1}{*}{5}  & 1.9700 & 3.2328 & 2.3690 & 2.3487
						& 1.9700 & 3.2323 & 2.3724 & 2.3522  \\
                             \hline
\multirow{1}{*}{10} & 1.9695 & 3.9156 & 6.8498 & 7.3299 
						 & 1.9695  & 3.9156 & 6.8517 & 7.3252  \\
                             \hline
\multirow{1}{*}{15} & 1.9695 & 3.9156 & 6.8715 & 7.3304
						&  1.9695 & 3.9156 & 6.8761 & 7.3223
						   \\
                             \hline
\end{tabular}
\caption{{Example 1}. Estimation of the order of convergence toward the reference stationary state for SP--CC scheme with explicit Euler (left) and RK4 (right) methods. Rates have been computed using $N=40,80$ and a reference solution  with $N=640$, $\sigma^2/2=0.1$, $\Delta t = \Delta w^2/2\sigma^2$. }
\label{tab:opinion}
\end{center}
\end{table}

\begin{table}
\begin{center}
\begin{tabular}{ c || c c c c  | c c c c }
\hline
 & \multicolumn{4}{c} {$SP-CC_k$}  &  \multicolumn{4}{c} {$SP-CC_k$} \\ \hline

             Time         &     2   &  4  &  6  & G &     2   &  4  &  6  & G    \\ 
             \hline
                                      
\multirow{1}{*} {1} & 1.0681 & 1.0648 & 1.0648 & 1.0648 
						& 2.0991 &  2.0931 & 2.0932 & 2.0932 \\
                             \hline
\multirow{1}{*}{5}  & 2.0093 & 1.9760 & 1.9648 & 1.9648
						& 2.0700 & 2.8755 & 2.6082 & 2.6092  \\
                             \hline
\multirow{1}{*}{10} & 2.0155 & 3.9714 & 3.4966 & 3.2983
						& 2.0700  & 4.0006 & 6.0768 & 7.4096  \\
                             \hline
\multirow{1}{*}{15} & 2.0155 & 3.9776 & 5.3045 & 7.3283 
						& 2.0700 & 3.9982  & 5.8780 & 9.0173\\
                             \hline
\end{tabular}
\caption{{Example 1}. Estimation of the order of convergence toward the reference stationary state for SP--CC scheme with first (left) and second order (right) semi--implicit methods. Rates have been computed using $N=40,80$ and a reference solution  with $N=640$, $\sigma^2/2=0.1$, $\Delta t = \Delta w/2\sigma^2$.}
\label{tab:opinion_SI}
\end{center}
\end{table}

In the general case $P(w,w_*)\ne 1$ and it is not possible to give an analytical formulation of the steady state solution $f^{\infty}(w,t)$. In Figure \ref{fig:BC} we represent a typical evolution of an aggregation model in the bounded confidence case \cite{PT2}
\be
P(w,w_*) = \chi(|w-w_*|\le \Delta),
\ee
where $\chi(\cdot)$ is the indicator function, for $\Delta = 0.4$, $\Delta = 0.8$. Here, the evolution has been computed through a  SP--CC with Gauss quadrature, the integral $\B[f](w,t)$ has been evaluated through a trapezoidal method. 

\begin{figure}
\centering
\subfigure[]{\includegraphics[scale=0.38]{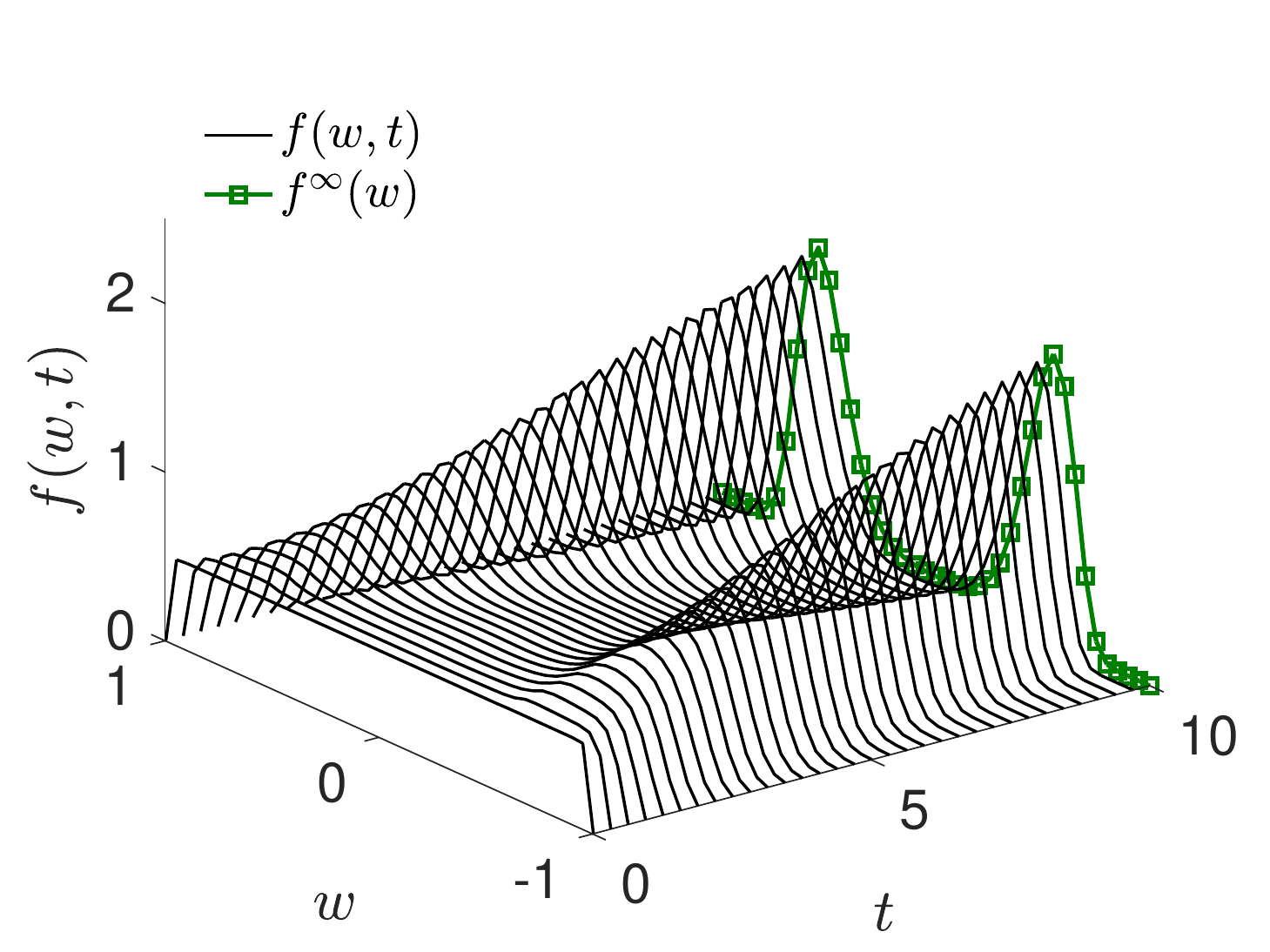}}
\subfigure[]{\includegraphics[scale=0.38]{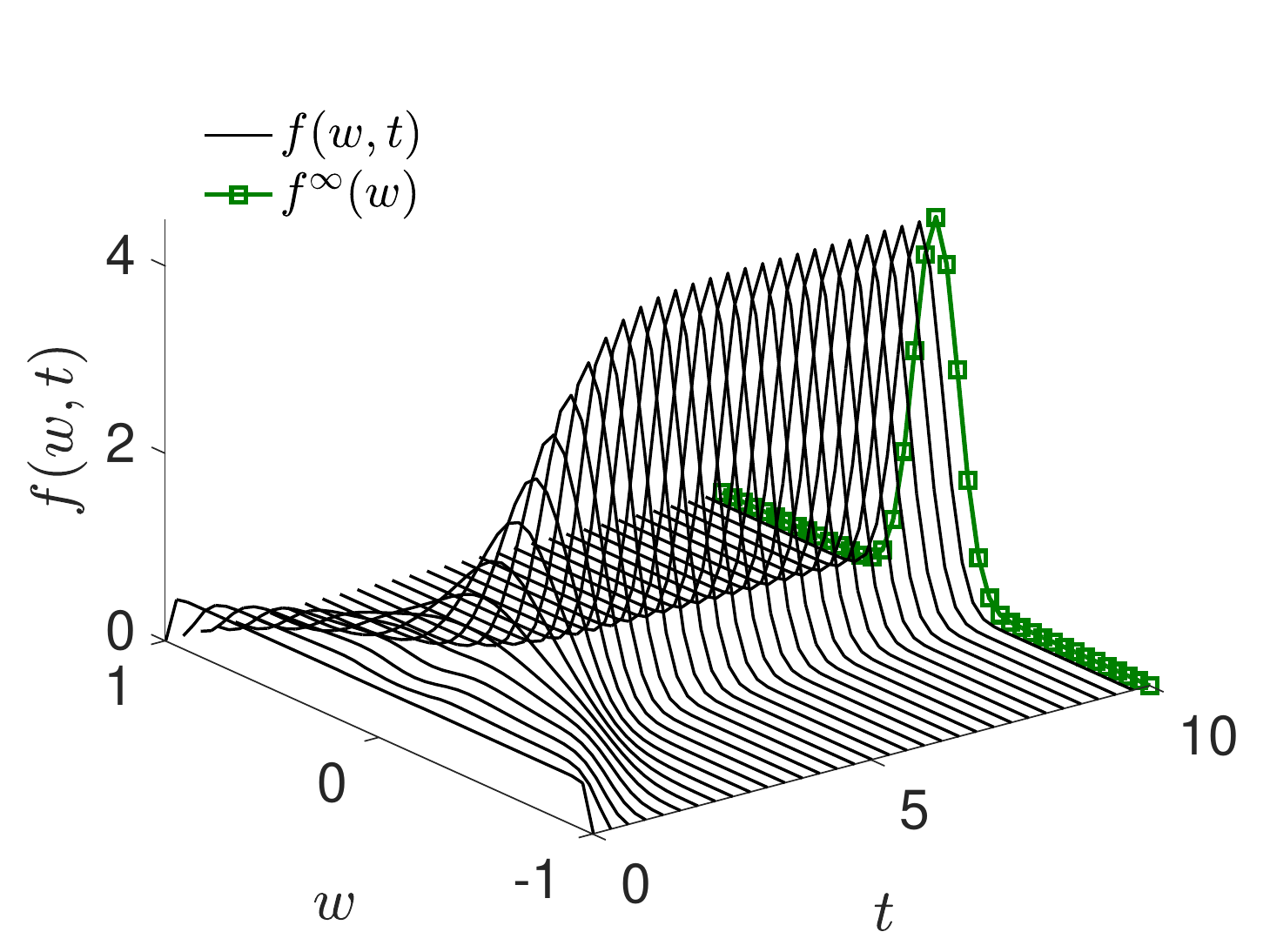}}
\caption{{Example 1}. Opinion model in the bounded confidence case with \textbf{(a)} $\Delta = 0.4$, \textbf{(b)} $\Delta = 0.8$. In both cases we considered $\Delta w=0.05$, $\sigma^2/2=0.01$, $\Delta t=\Delta w^2/2\sigma^2$. The reference stationary solution has been computed with $N = 640$ gridpoints.}
\label{fig:BC}
\end{figure}

\subsection{Example 2: Wealth evolution in unbounded domains}
Let us consider equation \eqref{eq:NAD} with $w\in\RR^+$ and
\be\label{eq:market_BD}
\B[f](w,t) = \int_{\RR^+} a(w,w_*)(w-w_*)f(w_*,t)dw_*, \qquad D(w) = \dfrac{\sigma^2}{2}w^2.
\ee
With the above choice, the Fokker-Planck equation describes the evolution of the wealth distribution $w$ at time $t$ in a large set of interacting economic agents (see \cite{NPT,PT2} for details).
 
In the case of constant interaction $a(w,w_*)\equiv 1$ the steady state of the equation is analytically computable 
\be\label{eq:market_anal}
f^{\infty}(w) = \dfrac{(\mu-1)^{\mu}}{\Gamma(\mu)w^{1+\mu}}\exp\left\{ -\dfrac{\mu-1}{w}  \right\},
\ee
where $\mu=1+2/\sigma^2$ is the so-called Pareto exponent. In the numerical test we consider the initial distribution
\be\label{eq:initial_market}
f(w,0) = \beta \left[ \exp\left(-c(w-u)^2\right) \right],\qquad c=20, 
\ee
with $\beta>0$ a normalization constant.

Again, due to degeneracy of the diffusion on the left boundary we report results only for SP--CC schemes.
In Figure \ref{fig:market_sigma} we present the solution with $u=1$ in the domain $[0,L]$, $L=10$. In both figures $a(\cdot,\cdot)= 1$ whereas the diffusion constant assumes different values. We report the evolution of the solution and the relative $L^1$ error with respect to the stationary state using $N=201$ points for the semi--implicit SP--CC scheme (SISP--CC). We observe how the introduced methods describe the stationary state with different levels of accuracy. Note that, at the right boundary we must introduce an artificial boundary condition in order to truncate the computational domain. In our numerical results we impose the quasi stationary condition \eqref{eq:quasi_SS} in order to evaluate $f_{N+1}(t)$, that is
\[
\dfrac{f_{N+1}(t)}{f_N(t)} = \exp\Big\{ -\int_{w_{N}}^{w_{N+1}} \dfrac{\B[f]+D(w)}{D(w)}dw \Big\}.
\]
In Table \ref{tab:market} we estimate the overall order of convergence of the semi--implicit SP--CC scheme for several integration methods with $N=51,101$ for the domain $[0,L]$, $L=10$, with reference solutions computed with $N=1601$ gridpoints. The time step is chosen in such a way that the CFL condition for the positivity of the scheme is satisfied, i.e. $\Delta t= O(\Delta w)$. We can observe that for short times the order of accuracy is limited by the semi--implicit method, which is first order accurate, whereas as we approach to the stationary solution the order depends on the quadrature formula used.  

\begin{figure}\label{fig:market_sigma}
\centering
\includegraphics[scale=0.4]{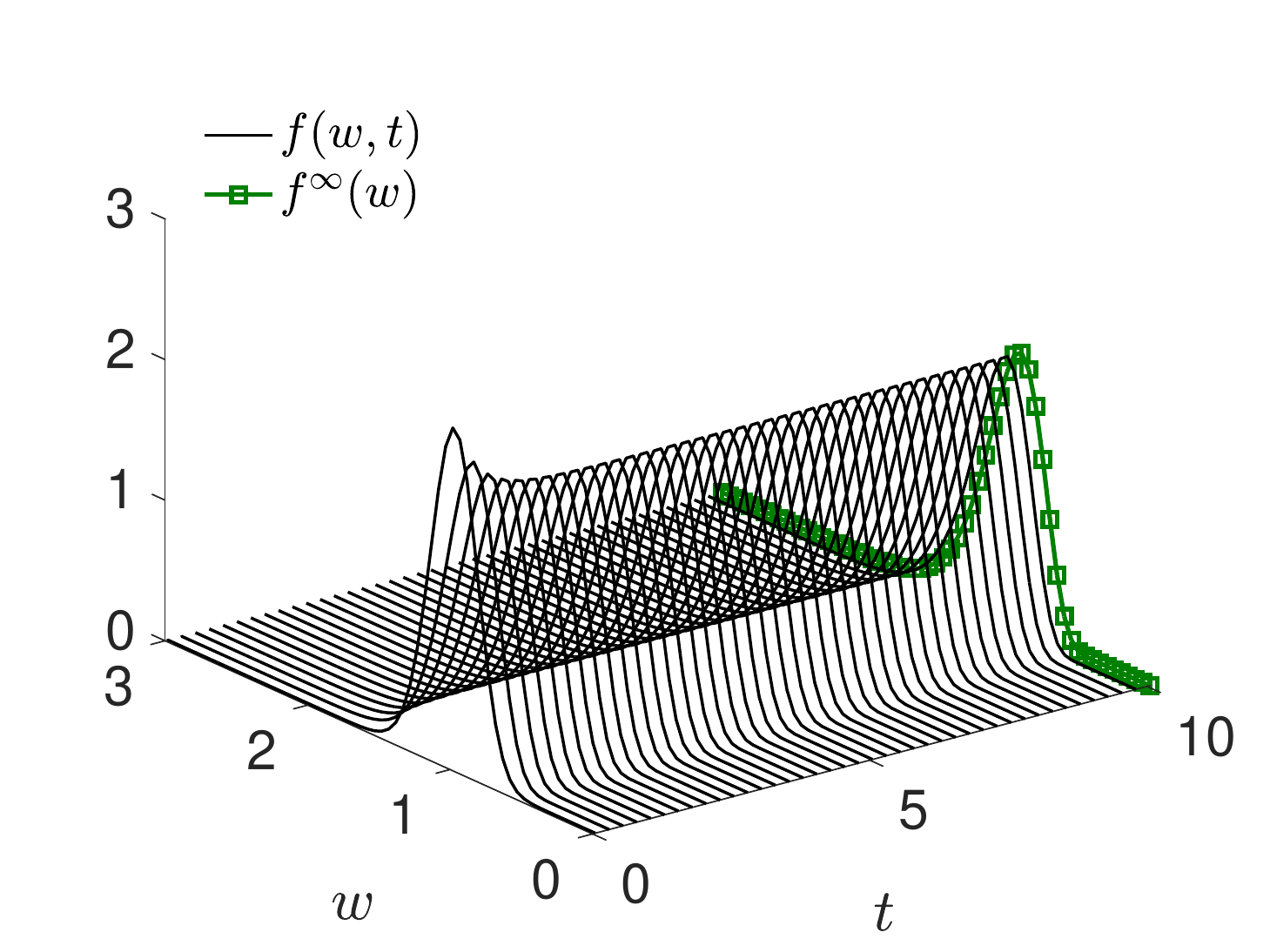}
\includegraphics[scale=0.295]{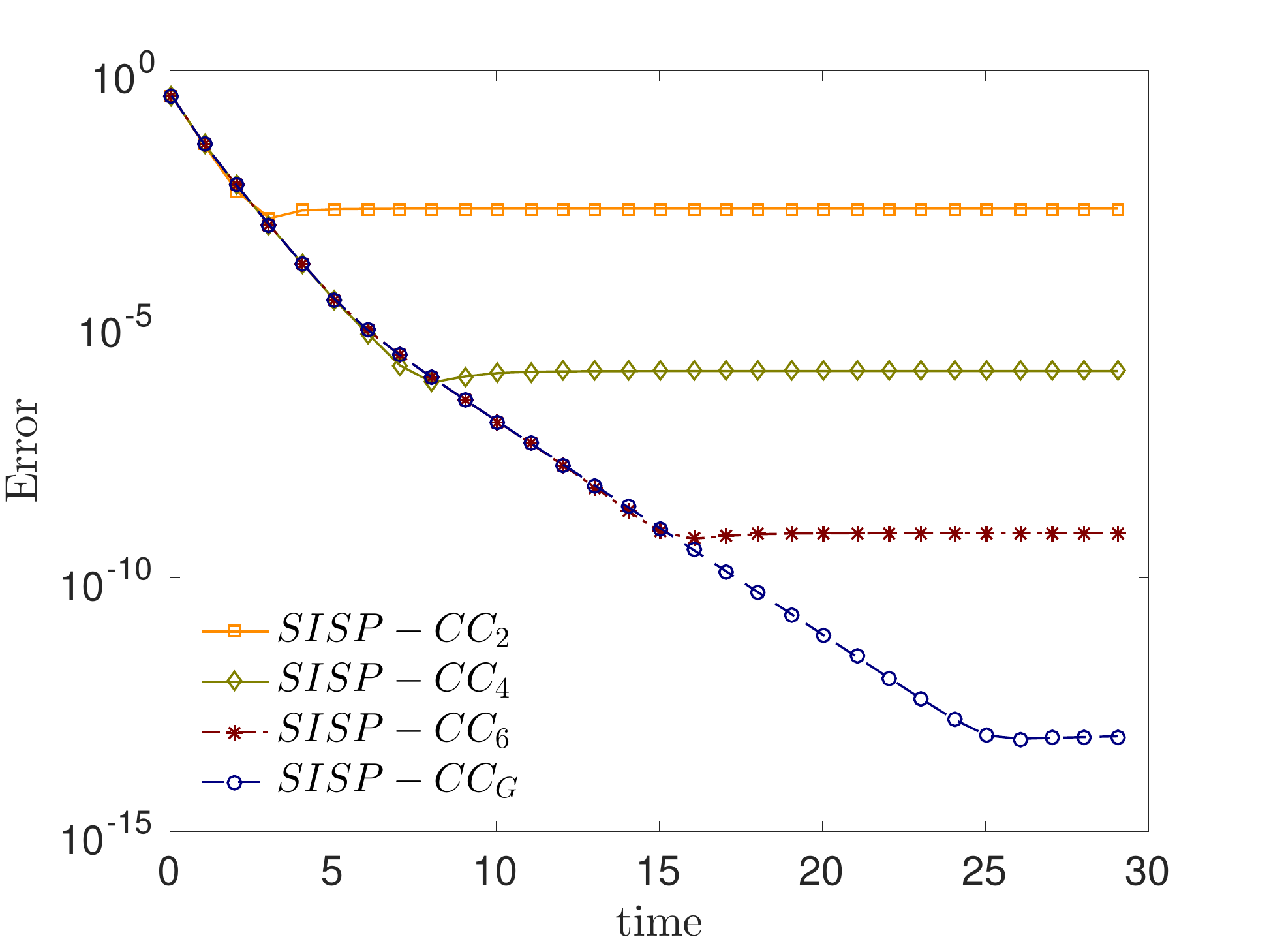}
\caption{{Example 2}. Left: evolution of the density $f(w,t)$ for \eqref{eq:market_BD} with $P(\cdot,\cdot)=1$, $u=1$, $\sigma^2/2=0.1$, $L=10$. In green we report the analytical steady state solution \eqref{eq:market_anal}. Right: evolution of the relative $L^1$ error for the different quadratures methods for the semi--implicit SP--CC scheme in the case $a(\cdot,\cdot)=1$, $\sigma^2/2=0.1$ and $\Delta w = 0.05$}\label{fig:market_1}
\end{figure}

\begin{table}
\begin{center}
\begin{tabular}{ c || c c c c  }
\hline
 & \multicolumn{4}{c}{$SP-CC_k$} \\
\hline

             Time                         & 2  &  4 & 6 & G  \\ \hline
                                      
\multirow{1}{*} {1} 
                             & 1.3047 & 1.5010 & 1.5021 & 1.5021 \\
                             \hline
\multirow{1}{*}{10}
                             & 1.9893 & 4.0634 & 2.8122 & 2.8682 \\
                             \hline
\multirow{1}{*}{20}
                             & 1.9894 & 3.9842 & 6.0784 & 10.0422\\
                             \hline
\end{tabular}
\caption{{Example 2}. Estimation of the order of convergence toward the reference stationary state for the semi--implicit SP-CC scheme, $N=51,101$, reference solution computed with $N=1601$, $\sigma^2/2=0.1$.}
\label{tab:market}
\end{center}
\end{table}

\subsection{Example 3: 2D model of swarming}
Let us consider a self-propelled swarming model of Cucker-Smale type \cite{BCCD} with diffusion. In this model the evolving distribution $f(x,w,t)$ represents the density of individuals (birds, fishes, $\ldots$) in position $x\in\RR^d$ having velocity $w\in\RR^d$ at time $t>0$. 
We have the following dynamics
\be\begin{split}\label{eq:model_2D}
\partial_t f(x,w,t)+w\nabla_x f(x,w,t) =& \nabla_w \cdot \Big[\alpha w(|w|^2-1)f(x,w,t)\\
&+(w-u_f)f(x,w,t)+D\nabla_wf(x,w,t)\Big],
\end{split}\ee
with 
\be\label{eq:K}
u_f(x,t) = \dfrac{\int_{\RR^{2d}}K(x,y)wf(y,w,t)dwdy}{\int_{\RR^{2d}}K(x,y)f(y,w,t)dwdy},
\ee
and $K(w,y)>0$ a localization kernel, $\alpha>0$ a self-propulsion term and $D>0$ a constant noise intensity. 

The space homogeneous version of the model \eqref{eq:model_2D} may be formulated in terms of the nonlinear Fokker--Planck equation \eqref{eq:NAD} with
\be\begin{split}\label{eq:BD_swarm}
\B[f](w,t) &= \alpha w(|w|^2-1)+ \int_{\RR^2} P(w,w_*)(w-w_*)f(w_*,t)dw_* ,\\
D(w)& = D,
\end{split}\ee
with $\alpha$ a positive constant and $P(w,w_*)\equiv 1$. The above equation can be written as a gradient flow. In fact, if we define
\be
\xi(w,t) = \Phi(w) + (U*f)(w,t) + D\log f(w,t),
\ee 
with $U(w)$ a Coloumb potential and $\Phi(w)$ a confining potential given by
\be
\Phi(w) = \alpha \left(\dfrac{|w|^4}{4}-\dfrac{|w|^2}{2}\right),
\ee
the equation reads
\be
\partial_t f(w,t) = \nabla_w \cdot \left(f(w,t) \nabla_w \xi(w,t)\right), \qquad w\in\RR^2.
\ee
A free energy functional which dissipates along solutions is defined by
\[
\mathcal E(t) = \int_{\RR^2} \left( \alpha\dfrac{|w|^4}{4}+(1-\alpha)\dfrac{|w|^2}{2} \right)f(w,t)dw-\dfrac{1}{2}|u_f|^2+D\int_{\RR^2} f(w,t)\log f(w,t) dw,
\]
with 
\[
u_f(t) = \dfrac{\int_{\RR^2}w f(w,t)dw}{\int_{\RR^2}f(w,t)dw}. 
\]
Stationary solutions should satisfy the identity $\nabla_w \xi= 0$ and have the form
\be\label{eq:stationary_swarming}
f^{\infty}(w) = C \exp\Bigg\{ -\dfrac{1}{D}\left[ \alpha \dfrac{|w|^4}{4}+(1-\alpha)\dfrac{|w|^2} {2}-u_{f^{\infty}}\cdot w \right] \Bigg\},
\ee
with $C>0$ a normalization constant. It is possible to prove the following result (see \cite{BCCD} for more details).
\begin{theorem}\label{th:1c}
Let us consider equation \eqref{eq:model_2D} in the space-homogeneous case, i.e. \eqref{eq:NAD} with $\B[f](w,t)$ and diffusion as in \eqref{eq:BD_swarm}, exhibits a phase transition in the following sense
\begin{itemize}
\item[i)] For small enough diffusion coefficient $D>0$ there is a function $u=u(D)$ with $\lim_{D\rightarrow 0}u(D)=1$, such that $f^{\infty}(w)$ with $u=(u(D),0,\dots,0)$ is a stationary solution of the original problem. 
\item[ii)] For large enough diffusion coefficients $D>0$ the only stationary solution is the symmetric distribution given by \eqref{eq:stationary_swarming} with $u_f \equiv 0$. 
\end{itemize}
\end{theorem}

\begin{table}
\begin{center}
\begin{tabular}{ c || c c c c | c c c c }
\hline
$\alpha=0$ & \multicolumn{4}{c} {$SP-CC_k$}  &  \multicolumn{4}{c} {$SP-EA_k$} \\ \hline

             Time         &     2   &  4  &  6  & G &     2   &  4  &  6  & G    \\ 
             \hline
                                      
\multirow{1}{*} {1} & 2.1387 &  2.1387 & 2.1387 & 2.1387 
                             & 2.4142 & 2.4142 & 2.4142 & 2.4142  \\
                             \hline
\multirow{1}{*}{5}  & 6.9430 & 6.9430 & 6.9430 & 6.9430  
                             & 10.0712 & 10.0712 & 10.0712 & 10.0712 \\
                             \hline
\multirow{1}{*}{10}& 20.0127  & 20.0127 & 20.0127 & 20.0127  
                             & 23.9838 & 23.9838 & 23.9838 & 23.9838  \\
                             \hline
\hline
$\alpha=1$ & \multicolumn{4}{c} {$SP-CC_k$}  &  \multicolumn{4}{c} {$SP-EA_k$} \\ \hline

             Time         &     2   &  4  &  6  & G &     2   &  4  &  6  & G    \\ 
             \hline
                                      
\multirow{1}{*} {1} & 2.5310 &  2.5310 & 2.5310 & 2.5310
                             & 2.2614 & 2.2892 & 2.2892 & 2.2892  \\
                             \hline
\multirow{1}{*}{5}  & 2.0498 & 7.6659 & 7.6659 & 7.6659  
                             & 2.0635 & 10.9818 & 10.9818 & 10.9818 \\
                             \hline
\multirow{1}{*}{10}& 2.0503  & 18.7697 & 18.7697 & 18.7697  
                             & 2.0613 & 14.8321 & 14.8321 & 14.8321  \\
                             \hline
\end{tabular}
\caption{{Example 3}. Estimation of the order of convergence for the one-dimensional swarming model for the explicit SP--CC and SP--EA over the domain $[-L,L]$ with $L=5$, $N=21,41,81$, $D=0.4$, $\Delta t = \Delta w^2/L^2$. }
\label{tab:swarming}
\end{center}
\end{table}

Since diffusion is constant, we compute the solution both using SP--CC type schemes and the entropic average schemes SP--EA. 
In Table \ref{tab:swarming} we estimate the order of convergence of the SP--CC and SP--EA schemes in the 1D case for several integration methods. We can observe how each method reach spectral accuracy in the case $\alpha=0$, in this case, in fact, all quadrature methods become exact being  the quantity $(\B[f]+D')/D$ a first order polynomial in $w$.

\begin{figure}
\centering
\includegraphics[scale=0.295]{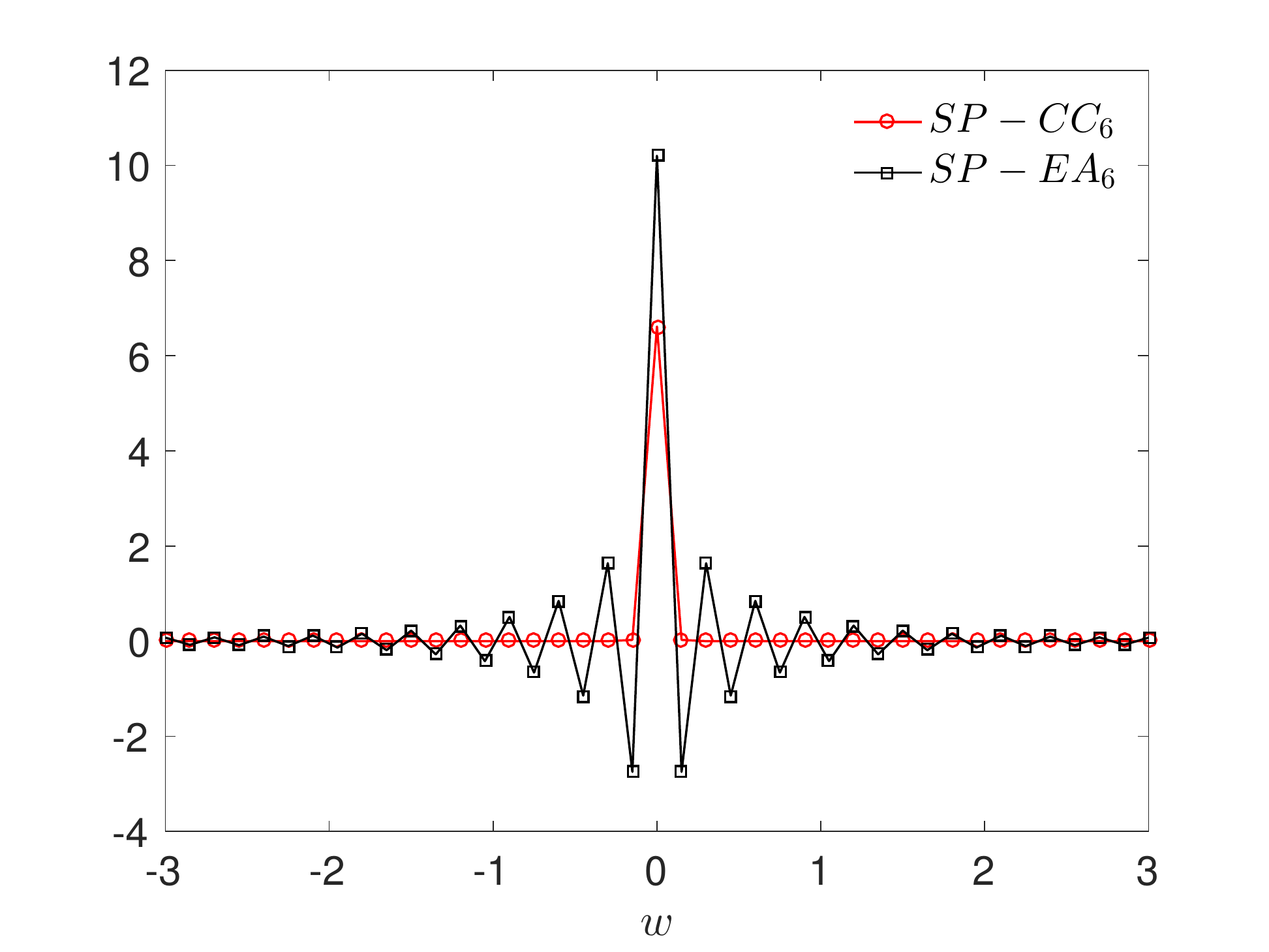}
\includegraphics[scale=0.295]{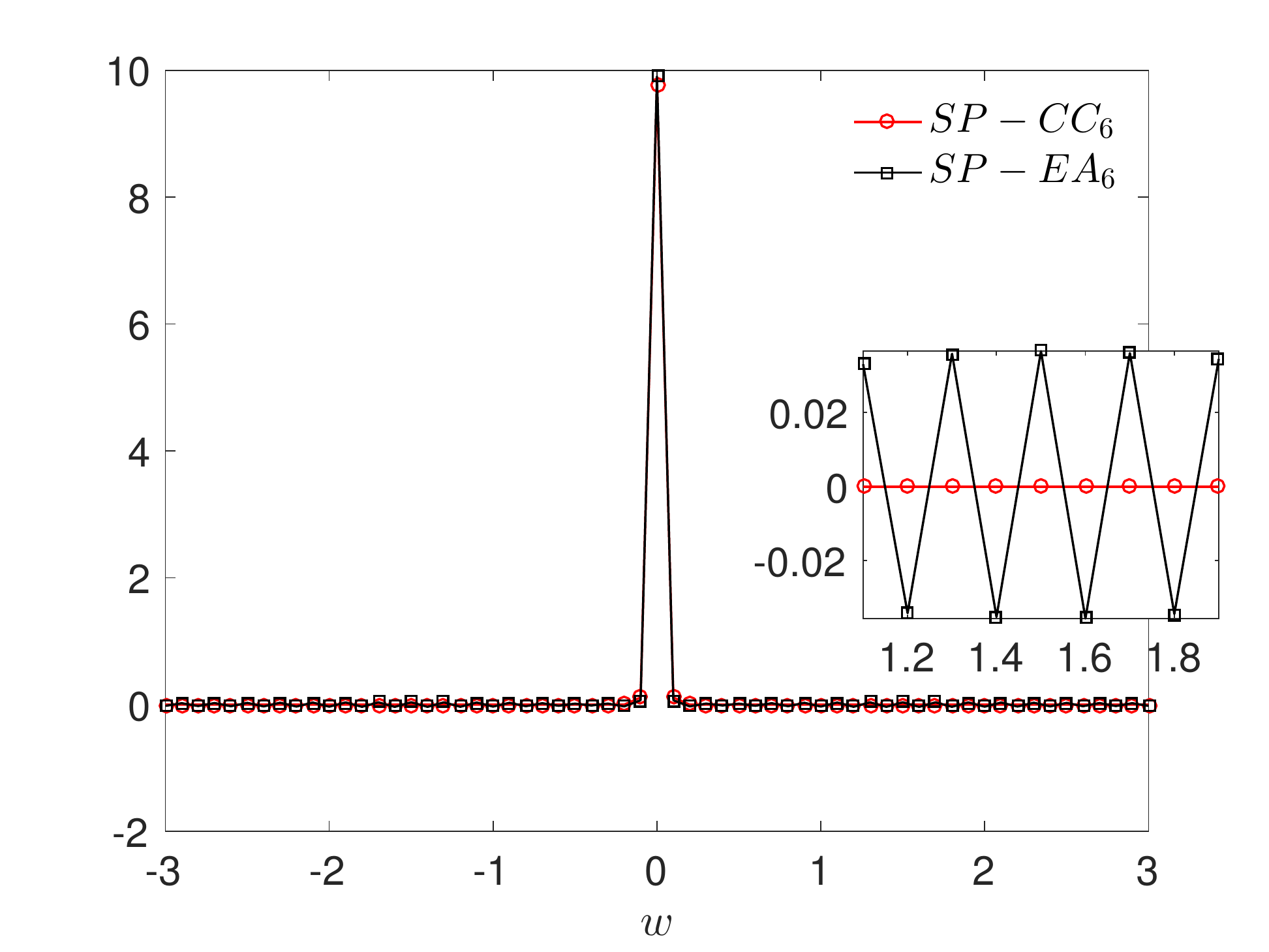}
\caption{{Example 3}. Stationary solution for the one-dimensional swarming model with $\alpha=1$ and $D=0.001$, $N=41$ (left) and $N=61$ (right). As expected the SP--EA scheme produces instabilities for a vanishing diffusion. The SP--CC scheme remains stable and first order accurate. }\label{fig:instability}
\end{figure}

In Figure \ref{fig:instability} we show that, as expected, on a coarse grid the SP--EA method becomes unstable for vanishing diffusions, whereas the SP--CC scheme remains stable and reduces to first order upwinding. In this case the solution becomes close to a Dirac delta in the velocity space. 
Finally, in Figure \ref{fig:2D_swarming} we present the resulting 2D nonlinear Fokker--Planck equation for swarming with $\B[f](w,t)$ and $D(w)$ in \eqref{eq:BD_swarm}, for several values of the diffusion coefficient $D=0.1,0.3,0.5$, and of the self-propulsion $\alpha=0,2,4$. The initial distribution is a bivariate normal distribution of the form
\[
f_0(w,v,0) = \dfrac{1}{2\pi\sqrt{\sigma^2_v\sigma^2_w}}\exp \Big\{ -\dfrac{1}{2} \Big[ \dfrac{(w-\mu_w)^2}{\sigma_w^2}+\dfrac{(v-\mu_v)^2}{\sigma_v^2} \Big] \Big\}
\] 
with $\mu_w=\mu_v=2$ and $\sigma_w^2=\sigma_v^2=0.5$. The generalization of the schemes to the multidimensional case is done dimension by dimension and is summarized in Appendix B. The semi-implicit numerical scheme has been used, with a $6$th order open Newton--Cotes quadrature method. It is possible to observe the threshold phenomenon occurring for an increasing diffusion prescribed by Theorem \ref{th:1c}. The results obtained with the two different schemes are essentially equivalent in this case.

\begin{figure}
\centering
\subfigure[$\alpha=0; D=0.1$]{\includegraphics[scale=0.24]{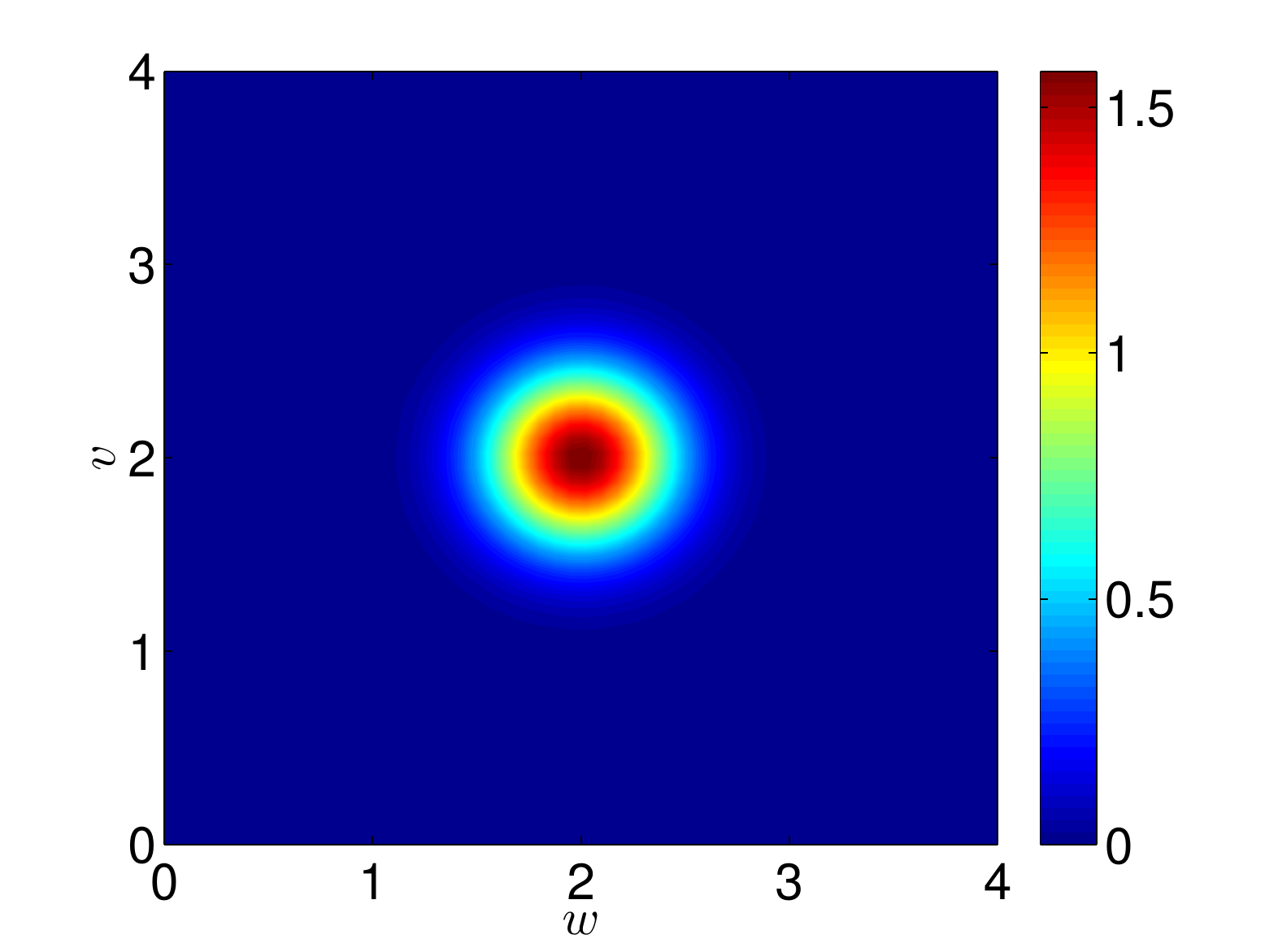}}
\subfigure[$\alpha=0; D=0.3$]{\includegraphics[scale=0.24]{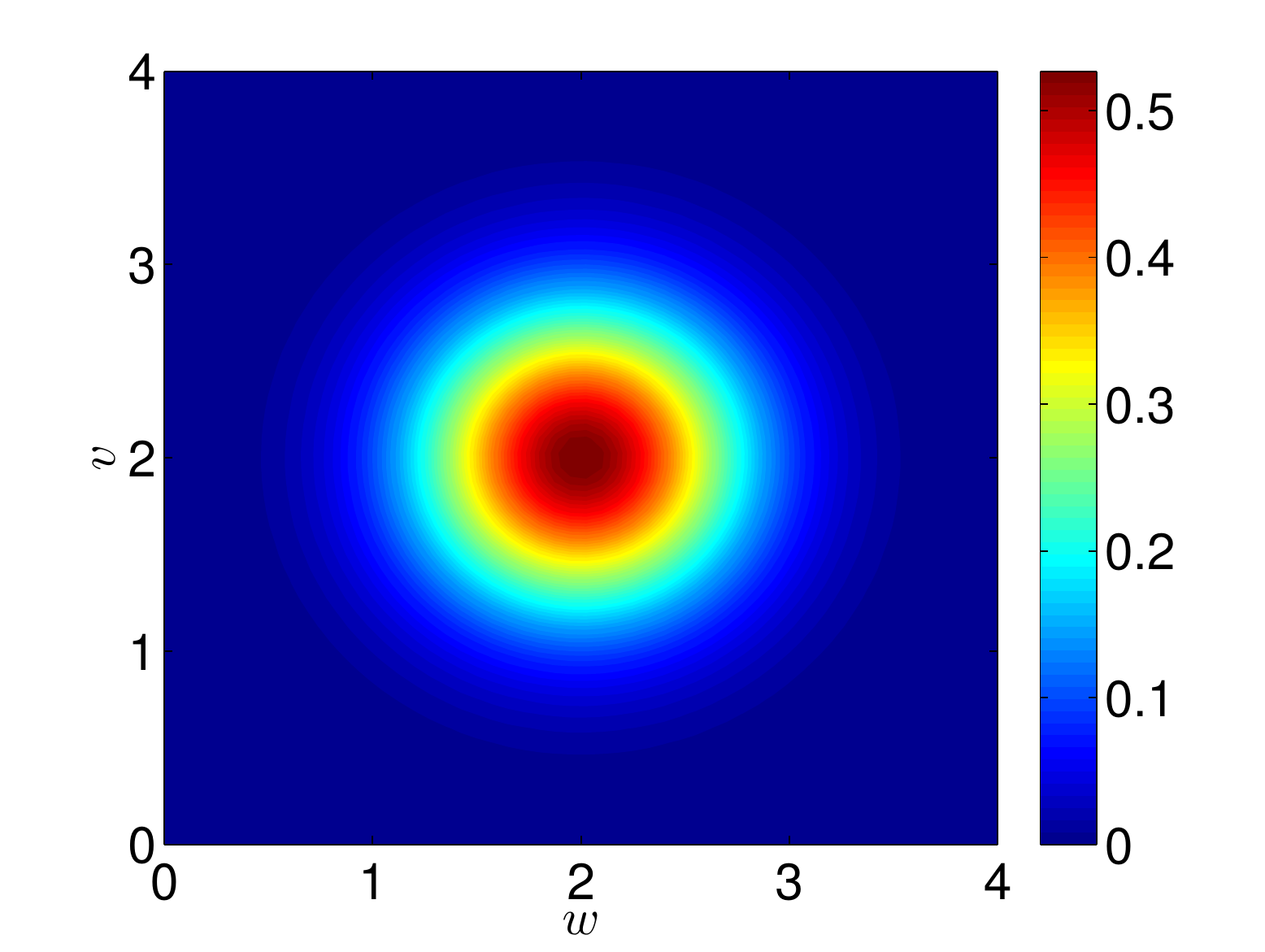}}
\subfigure[$\alpha=0; D=0.5$]{\includegraphics[scale=0.24]{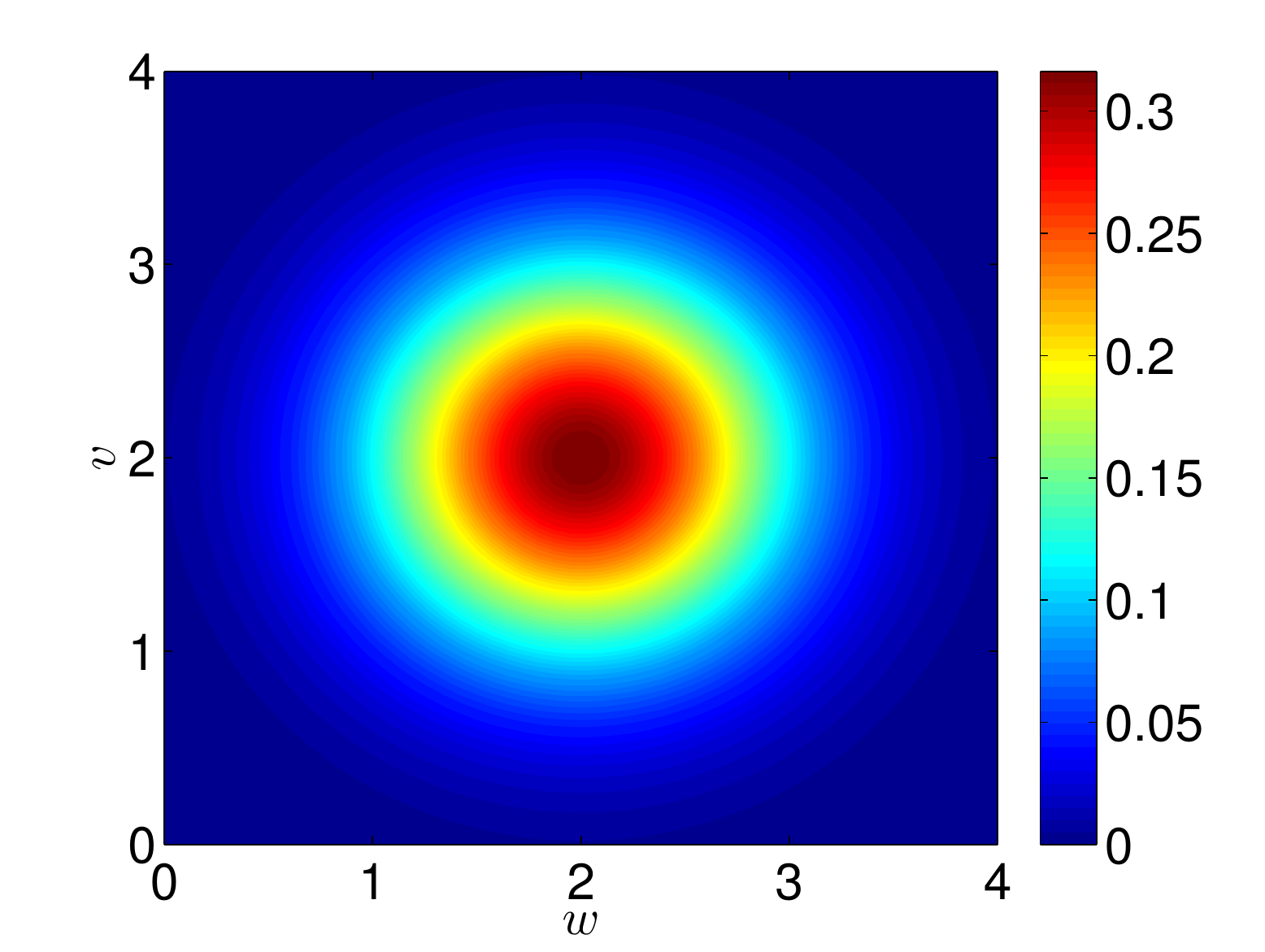}}\\
\subfigure[$\alpha=2; D=0.1$]{\includegraphics[scale=0.24]{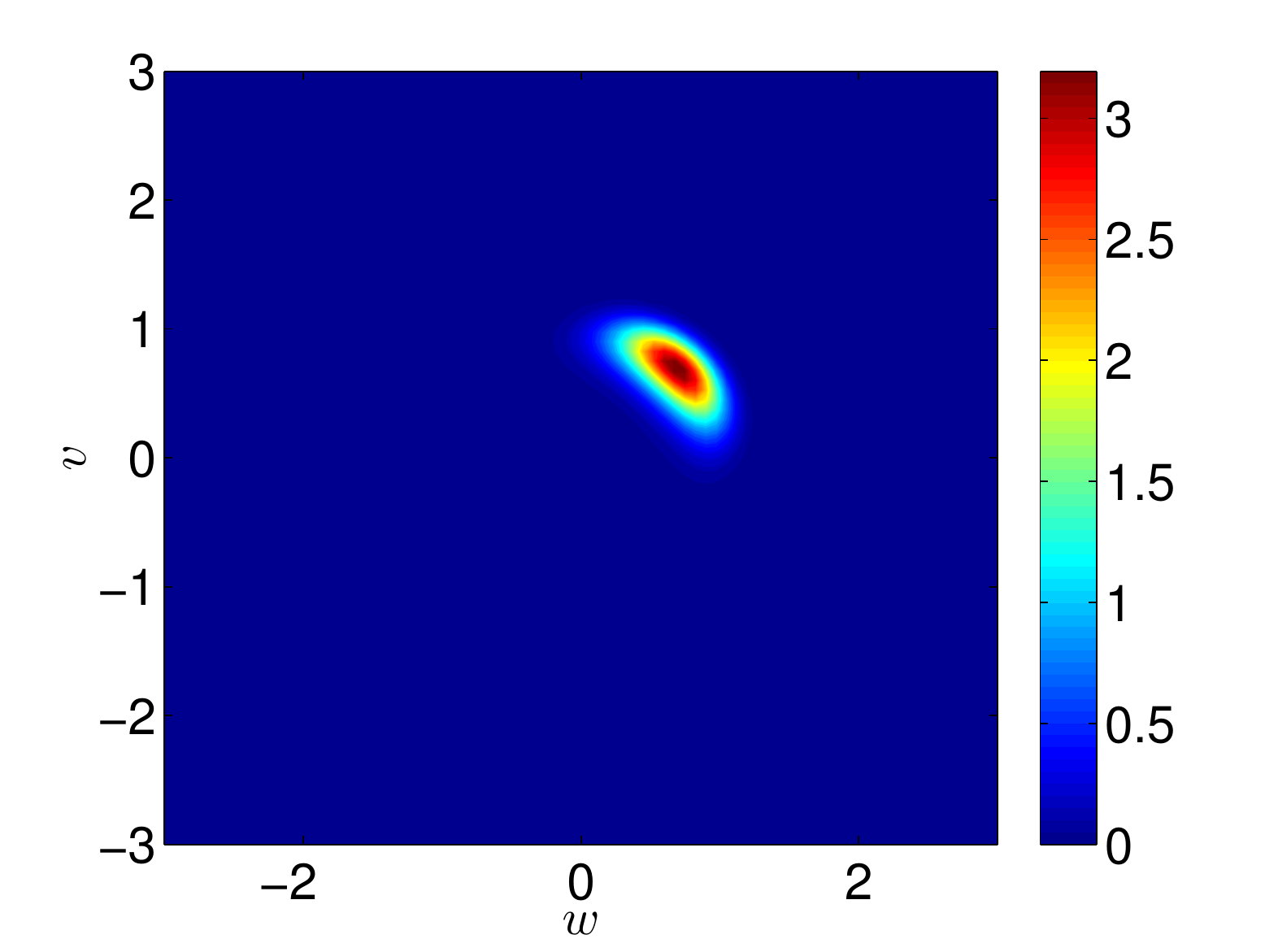}}
\subfigure[$\alpha=2; D=0.3$]{\includegraphics[scale=0.24]{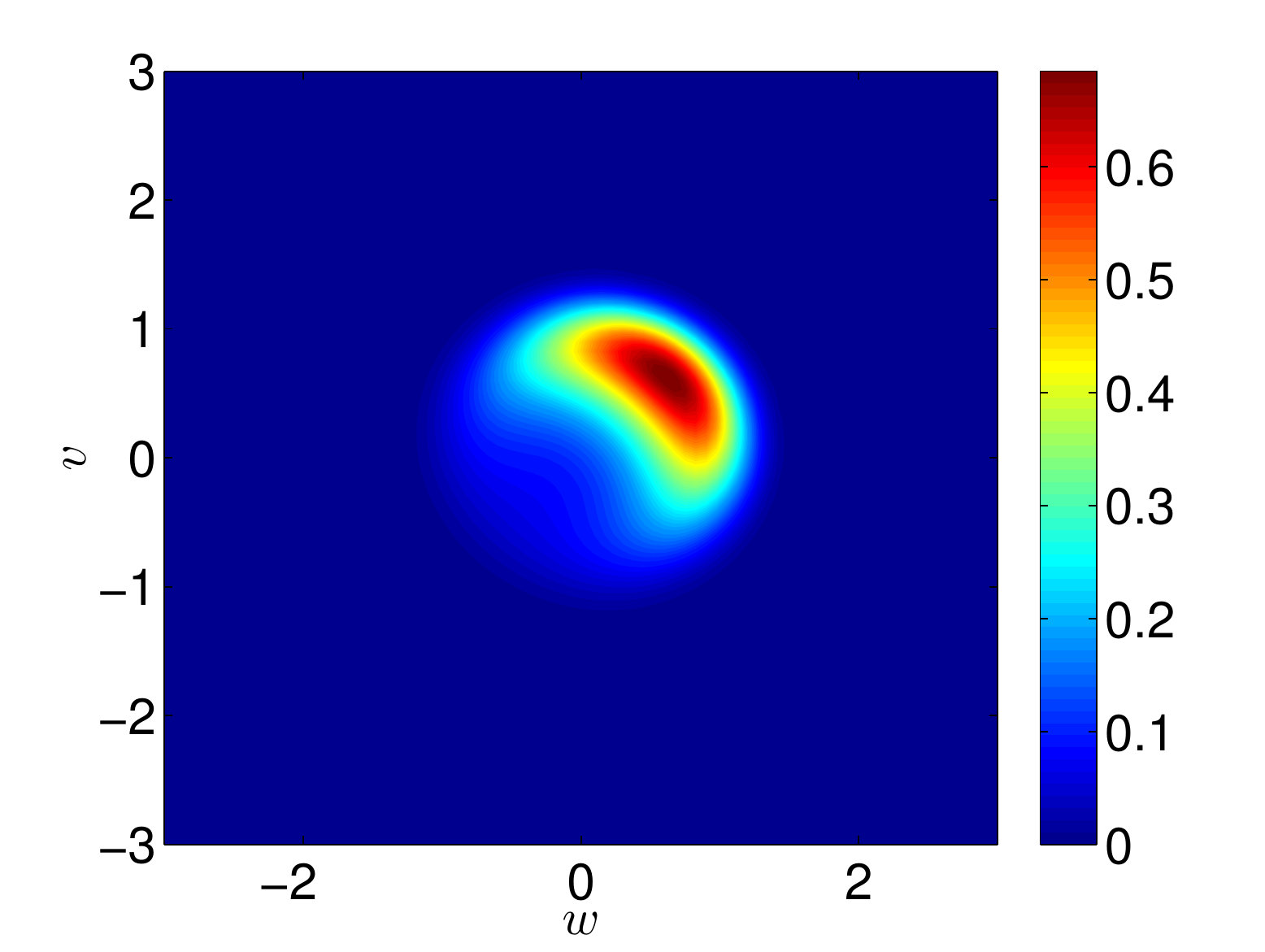}}
\subfigure[$\alpha=2; D=0.5$]{\includegraphics[scale=0.24]{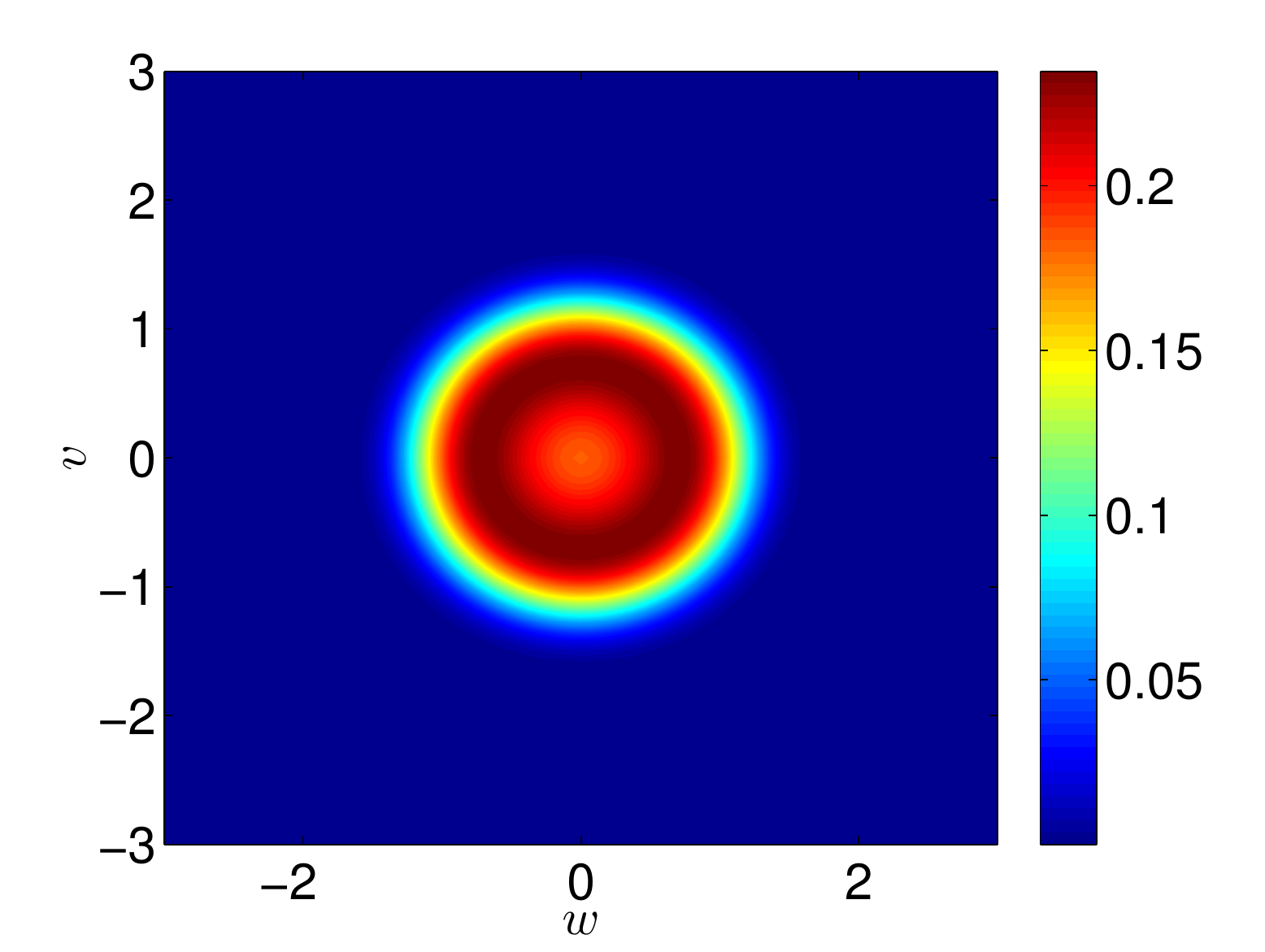}}\\
\subfigure[$\alpha=4; D=0.1$]{\includegraphics[scale=0.24]{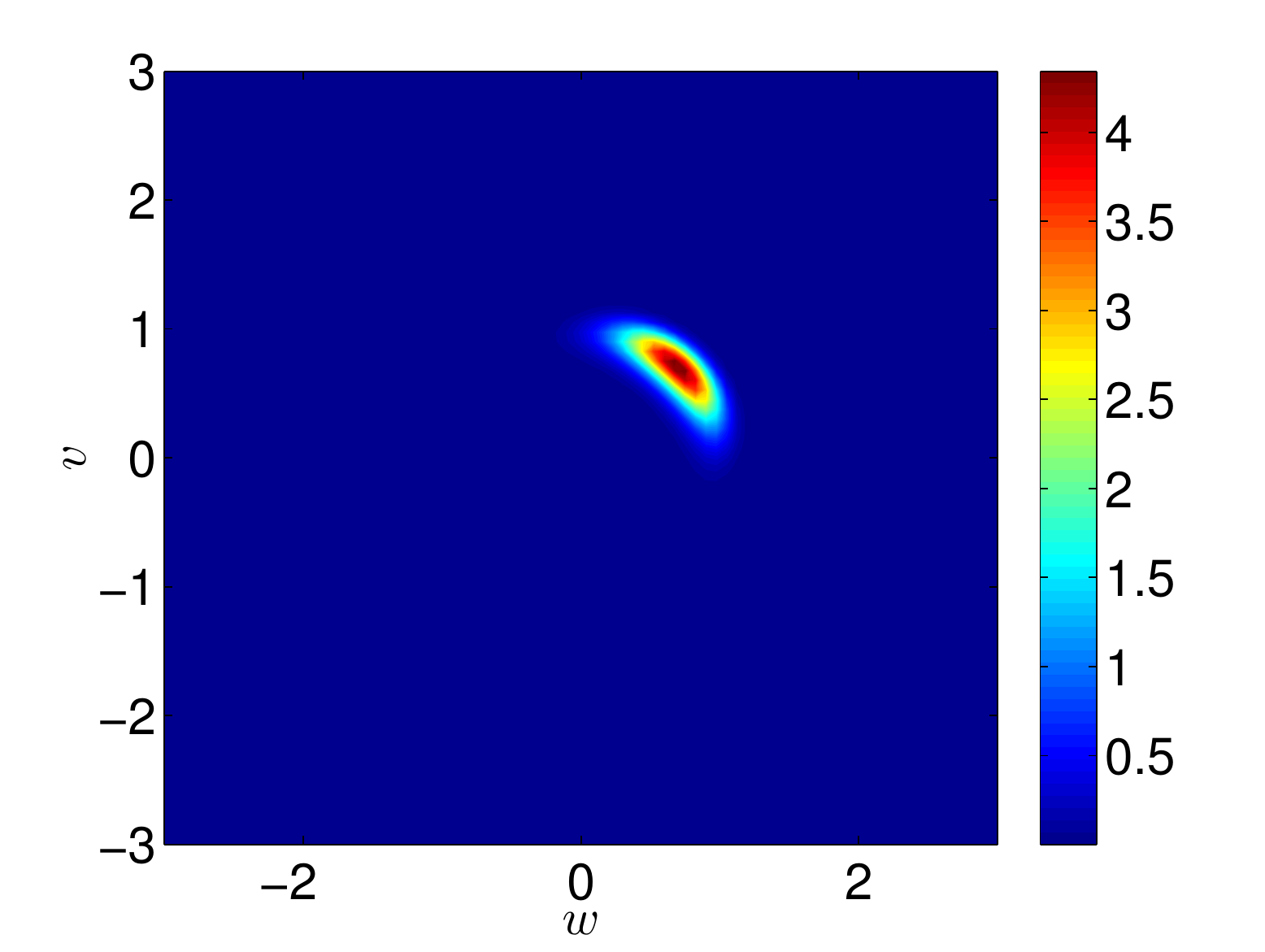}}
\subfigure[$\alpha=4; D=0.3$]{\includegraphics[scale=0.24]{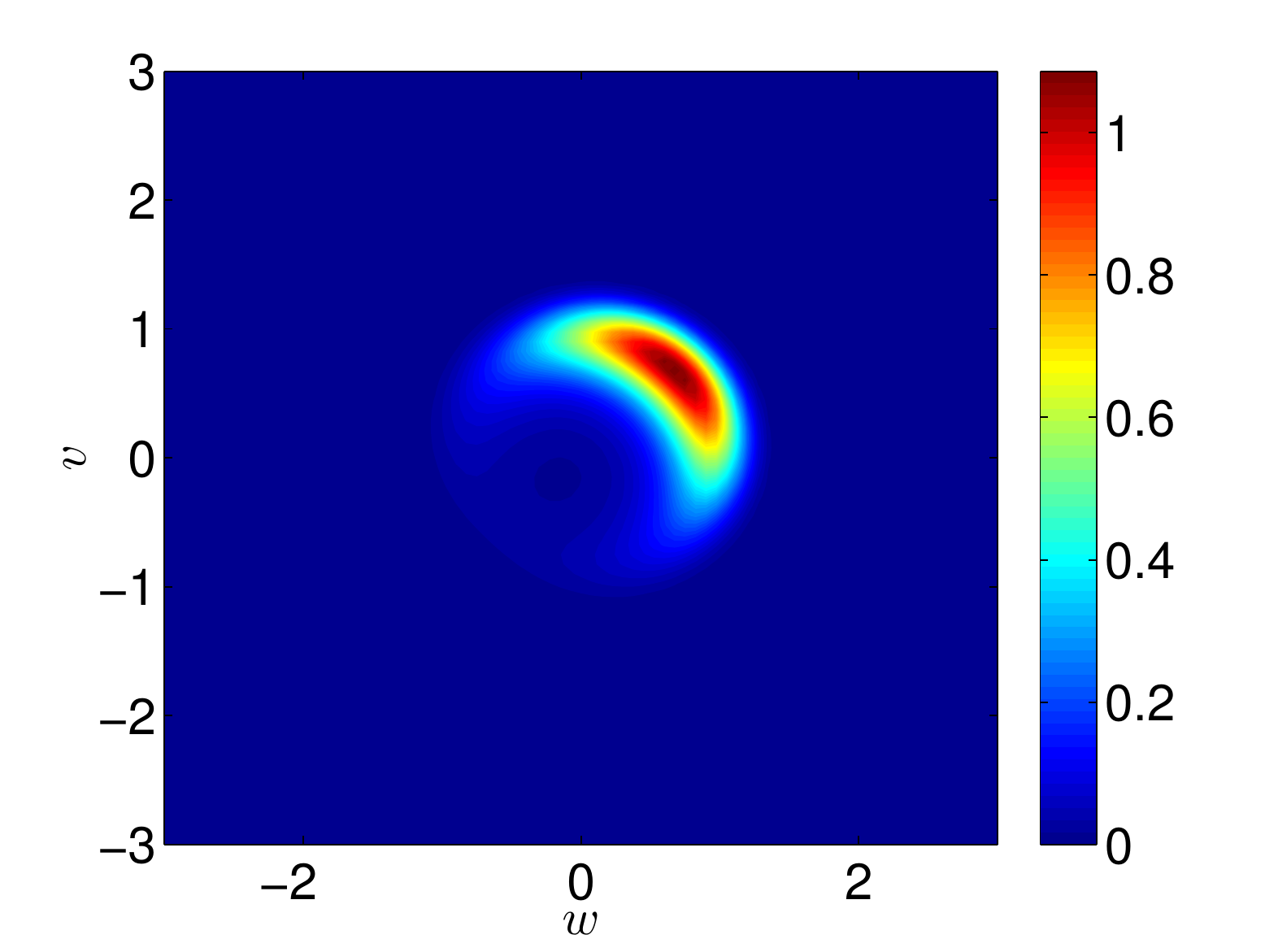}}
\subfigure[$\alpha=4; D=0.5$]{\includegraphics[scale=0.24]{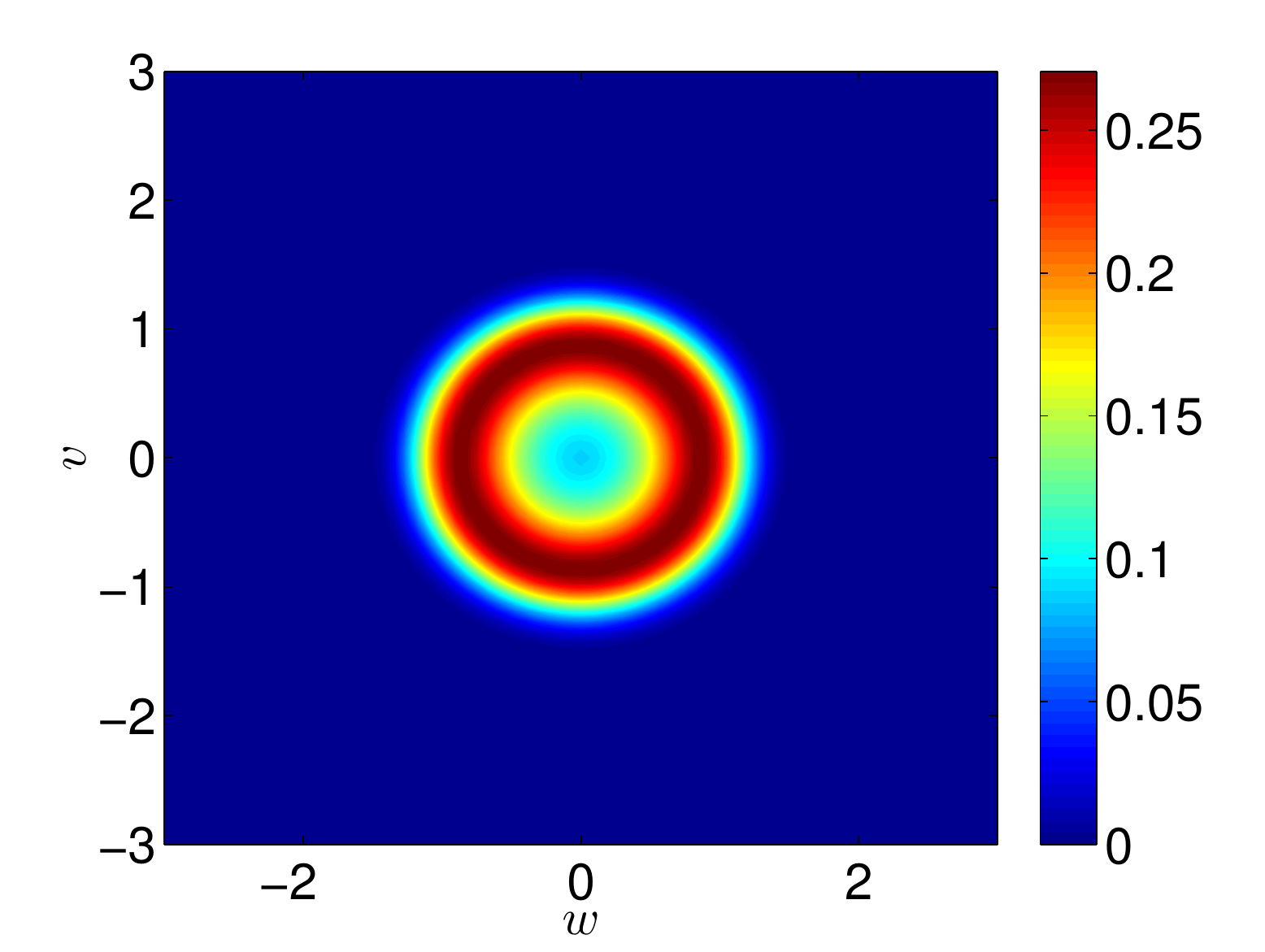}}
\caption{{Example 3}. Stationary state of the two-dimensional swarming model for several values of the diffusion coefficient $D>0$ and fixed self-propulsion $\alpha=0,2,4$. In the case $\alpha>0$ for increasing values of $D$ the mean of the stationary distribution approaches to $[0,0]$. We considered a discretization $(w,v)\in[-L,L]\times[-L,L]$, $L=3$, $\Delta w=\Delta v=0.05$ and $\Delta t = dw/L$.}
\label{fig:2D_swarming}
\end{figure}

\section*{Conclusion}
The construction of structure--preserving schemes for nonlinear Fokker--Planck equations has been studied.  Two different types of schemes have been constructed. The first type represents a natural extension of the so--called Chang--Cooper scheme to the nonlinear case. The second type of schemes represents a modification with better entropy dissipation properties. 
Both methods are second order accurate and capable to preserve the stationary state with arbitrary accuracy.
However, non negativity restrictions are more severe for the second type of schemes. 
Even if the analysis is performed in the one-dimensional case, extensions to multidimensional situations are also considered. Several applications to linear and nonlinear Fokker-Planck equations arising in socio-economic sciences are presented and show the generality of the present approach. Extensions of the schemes to include nonlinear diffusion terms and higher order schemes in the limiting of vanishing diffusion are actually under study and will be presented elsewhere.

\section*{Acknowledgements}
The research that led to the present paper was partially supported by the research grant {\it Numerical methods for uncertainty quantification in hyperbolic and kinetic equations} 
of the group GNCS of INdAM. MZ acknowledges ``Compagnia di San Paolo''. 

\appendix

\section{High order semi-implicit methods}\label{appendix:A}
Here we follow the approach in \cite{BFR}. We write the semi-discrete scheme \eqref{eq:dflux} in the equivalent form 
\be\label{eq:sys}
\begin{split}
\frac{df_i}{dt}(w,t) &= {\mathcal Q}_i({\bf f},{\bf g}),\\
\frac{dg_i}{dt}(w,t) &= {\mathcal Q}_i({\bf f},{\bf g}),\\
\end{split}
\ee
where ${\bf f}=(f_0,\ldots,f_N)$, ${\bf g}=(g_0,\ldots,g_N)$, 
\be
{\mathcal Q}_i({\bf f},{\bf g}) = \frac{{\F}_{i+1/2}[{\bf f},{\bf g}]-{\F}_{i-1/2}[{\bf f},{\bf g}]}{\Delta w}
\ee
and
\be
{\F}_{i+1/2}[{\bf f},{\bf g}]= \tilde{\C}_{i+1/2}[{\bf g}] \left[ (1-\delta_{i+1/2}[{\bf g}])f_{i+1}+\delta_{i+1/2}[{\bf g}] f_i \right]+D_{i+1/2}\dfrac{f_{i+1}-f_i}{\Delta w},
\ee
with initial conditions $f_i(0)=f_0(w_i)$, $g_i(0)=f_0(w_i)$. In the above equations we used the notation $[\cdot]$ to denote the functional dependence.\\
System (\ref{eq:sys}) is then solved by an implicit-explicit (IMEX) Runge-Kutta method \cite{PR} where the variables ${f}_i$ are treated implicitly and the variables ${g}_i$ are treated explicitly. 
More precisely, using standard notations we can write an implicit-explicit Runge-Kutta scheme for (\ref{eq:sys}) as follows. First we set $f_i^n=g_i^n$ and compute for $h=1,\ldots,s$
\be
\left\lbrace
\begin{split}
F_i^h&=f_i^n+\Delta t \sum_{k=1}^{h} a_{hk} {\mathcal Q}_i({\bf F}^k,{\bf G}^k) ,\\
G_i^h&=f_i^n+\Delta t \sum_{k=1}^{h-1} {\tilde a}_{hk} {\mathcal Q}_i({\bf F}^k,{\bf G}^k),
\end{split}
\right.
\ee
where ${\bf F}^k = (F_0^k,\ldots,F_N^k)$, ${\bf G}^k = (G_0^k,\ldots,G_N^k)$  and next we update the numerical solution 
\be
\left\lbrace
\begin{split}
f_i^{n+1}&=f_i^n+\Delta t \sum_{k=1}^{s} b_k {\mathcal Q}_i({\bf F}^k,{\bf G}^k) ,\\
g_i^{n+1}&=f_i^n+\Delta t \sum_{k=1}^{s} {\tilde b}_{k} {\mathcal Q}_i({\bf F}^k,{\bf G}^k).
\end{split}
\right.
\ee
In particular, the IMEX scheme is chosen such that $b_k={\tilde b}_k$, $k=1,\ldots,s$ so that $f^{n+1}=g^{n+1}$ and therefore the duplication of the system is only apparent since there is only one set of numerical solutions. In our numerical tests we coupled the structure preserving discretizations with the second order semi-implicit scheme obtained as a combination of Heun method (explicit) and Crank-Nicolson (implicit) characterized by $s=2$ and
\be
a_{11}=0,\quad a_{21}=a_{22}=1/2,\quad \tilde a_{21}=1,\quad b_k={\tilde b}_k=1/2,\quad k=1,2.
\ee
As observed in \cite{BFR} this represents a natural
choice when dealing with convection-diffusion type equations, since the Heun method is an SSP explicit
RK method, and Crank-Nicolson is an A-stable method, widely used for diffusion problems. Higher order methods can be found in \cite{BFR,PR}.

\section{The multi-dimensional case}\label{sec:2D}
In this section we report for the sake of completeness the details of the numerical schemes for multi-dimensional situations. We consider the case of Chang-Cooper type fluxes, and to keep notations simple we restrict to two dimensional problems $d=2$. We introduce a uniform mesh $(w_i,v_j) \in\Omega\subseteq \RR^2$, with $\Delta w=w_{i+1}-w_{i}$ and $\Delta v=v_{j+1}-v_{j}$. We denote by $w_{i+1/2}=w_i+\Delta w/2$ and $v_{j+1/2}=v_j+\Delta v/2$. Let $f_{i,j}(t)$ be an approximation of the solution $f(w_i,v_j,t)$ and consider the following discretization of the nonlinear Fokker-Planck equation \eqref{eq:NAD_dim}  
\be\label{eq:eqflux2D}
\frac{d}{dt} f_{i,j} = \dfrac{\F_{i+1/2,j}[f]-\F_{i-1/2,j}[f]}{\Delta w}+\dfrac{\F_{i,j+1/2}[f]-\F_{i,j-1/2}[f]}{\Delta v},
\ee
being $\F_{i\pm1/2,j}[f]$, $\F_{i,j\pm1/2}[f]$ flux functions characterizing the numerical discretization. The quasi-stationary approximations over the cell $[w_i,w_{i+1}]\times[v_i,v_{i+1}]$ of the two dimensional problem now read
\[
\begin{split}
\int_{w_{i}}^{w_{i+1}} \dfrac{1}{f(w,v_j,t)}\partial_w f(w,v_j,t) dw&= -\int_{w_{i}}^{w_{i+1}}\dfrac{\B[f](w,v_j,t)+\partial_w D(w,v_j)}{D(w,v_j)}dw, \\
\int_{v_{j}}^{v_{j+1}} \dfrac{1}{f(w_i,v,t)}\partial_v f(w_i,v,t) dv&= -\int_{v_j}^{v_{j+1}}\dfrac{\B[f](w_i,v,t)+\partial_v D(w_i,v)}{D(w_i,v)}dv.
\end{split}
\]
Therefore, setting
\be\begin{split}
\tilde{\C}_{i+1/2,j} &= \dfrac{D_{i+1/2,j}}{\Delta w}\int_{w_i}^{w_{i+1}}\dfrac{\B[f](w,v_j,t)+\partial_{w}D(w,v_j)}{D(w,v_j)}dw\\
\tilde{\C}_{i,j+1/2} &= \dfrac{D_{i,j+1/2}}{\Delta v}\int_{v_j}^{v_{j+1}}\dfrac{\B[f](w_i,v,t)+\partial_{v}D(w_i,v)}{D(w_i,v)}dv
\end{split}\ee
and by considering the natural generalization of the one-dimensional numerical flux
\be\begin{split}\label{eq:F2D}
\F_{i+1/2,j}[f] &= \tilde{\C}_{i+1/2,j}\tilde{f}_{i+1/2,j}+D_{i+1/2,j}\dfrac{f_{i+1,j}-f_{i,j}}{\Delta w}\\
\tilde f_{i+1/2,j}& = (1-\delta_{i+1/2,j})f_{i+1,j}+\delta_{i+1/2,j}f_{i,j}\\
\F_{i,j+1/2}[f] &= \tilde{\C}_{i,j+1/2}\tilde{f}_{i,j+1/2}+D_{i,j+1/2}\dfrac{f_{i,j+1}-f_{i,j}}{\Delta v}\\
\tilde f_{i,j+1/2}& = (1-\delta_{i,j+1/2})f_{i,j+1}+\delta_{i,j+1/2}f_{i,j},
\end{split}\ee
we define $\delta_{i+1/2,j}$ and $\delta_{i,j+1/2}$ in such a way that we preserve the steady state solution for each dimension. The Chang-Cooper type structure preserving methods are then given by
\be\begin{split}\label{eq:delta2D_1}
\delta_{i+1/2,j}     &= \dfrac{1}{\lambda_{i+1/2,j}}+\dfrac{1}{1-\exp(\lambda_{i+1/2,j})},\\
\lambda_{i+1/2,j} &= \dfrac{\Delta w \tilde{\C}_{i+1/2,j}}{D_{i+1/2,j}}
\end{split}\ee
and
\be\begin{split}\label{eq:delta2D_2}
\delta_{i,j+1/2}     &= \dfrac{1}{\lambda_{i,j+1/2}}+\dfrac{1}{1-\exp(\lambda_{i,j+1/2})},\\
\lambda_{i,j+1/2} &= \dfrac{\Delta v \tilde{\C}_{i,j+1/2}}{D_{i,j+1/2}}.
\end{split}\ee
The cases of higher dimension $d\ge 3$ and entropic average fluxes may be derived in a 
similar way.


\end{document}